\documentclass{amsart}
\usepackage{hyperref, amssymb, xypic, amsmath, mathtools, amsthm, enumerate}

\DeclareMathOperator{\Proj}{Proj}
\DeclareMathOperator{\Sym}{Sym}
\DeclareMathOperator{\diag}{diag}
\DeclareMathOperator{\tr}{tr}
\DeclareMathOperator{\adj}{adj}
\DeclareMathOperator{\Aut}{Aut}
\DeclareMathOperator{\colim}{colim}
\DeclareMathOperator{\Func}{Func}
\DeclareMathOperator{\Stone}{Stone}
\DeclareMathOperator{\Clopen}{Clopen}
\DeclareMathOperator{\Spec}{Spec}
\DeclareMathOperator{\Cont}{Cont}
\DeclareMathOperator{\RP}{RP}
\DeclareMathOperator{\Rann}{R}
\DeclareMathOperator{\inter}{int}
\DeclareMathOperator{\range}{range}
\DeclareMathOperator{\cl}{cl}
\DeclareMathOperator{\ActiveProj}{AProj}

\newcommand{\AWstar}{\cat{AWstar}}
\newcommand{\Wstar}{\cat{Wstar}}
\newcommand{\cAWstar}{\cat{cAWstar}}
\newcommand{\PAWstar}{\cat{pAWstar}}
\newcommand{\cOML}{\cat{COrtho}}
\newcommand{\CBoolean}{\cat{CBool}}
\newcommand{\PCBoolean}{\cat{pCBool}}
\newcommand{\Group}{\cat{Group}}
\newcommand{\PGroup}{\cat{pGroup}}
\newcommand{\cAL}{\cat{cActive}}
\newcommand{\AL}{\cat{Active}}
\newcommand{\EAWstar}{\cat{eAWstar}}
\newcommand{\cat}[1]{\ensuremath{\mathbf{#1}}}
\newcommand{\generated}[2]{\ensuremath{#1 \langle #2 \rangle}}

\newcommand{\op}{\ensuremath{^\mathrm{op}}}
\newcommand{\cC}{\ensuremath{\mathcal{C}}}
\newcommand{\commeas}{\ensuremath{\odot}}
\newcommand{\C}{\mathbb{C}}
\newcommand{\M}{\mathbb{M}}
\newcommand{\R}{\mathbb{R}}
\renewcommand{\Re}{\ensuremath{\mathrm{Re}}}
\newcommand{\ie}{\textit{i.e.}}

\numberwithin{equation}{section}
\theoremstyle{plain}
\newtheorem{theorem}[equation]{Theorem}
\newtheorem*{theorem*}{Theorem}
\newtheorem{lemma}[equation]{Lemma}
\newtheorem{proposition}[equation]{Proposition}
\newtheorem{corollary}[equation]{Corollary}
\newtheorem*{corollary*}{Corollary}
\theoremstyle{definition}
\newtheorem{definition}[equation]{Definition}
\newtheorem{example}[equation]{Example}
\newtheorem{remark}[equation]{Remark}

\begin{document}

\title{Active lattices determine AW*-algebras}
\author{Chris Heunen}
\address{Department of Computer Science, University of Oxford, \\ Wolfson
  Building, Parks Road, OX1 3QD, Oxford, UK}
\email{heunen@cs.ox.ac.uk}
\author{Manuel L. Reyes}
\address{Department of Mathematics, Bowdoin College, \\
  8600 College Station, Brunswick, ME 04011, USA}
\email{reyes@bowdoin.edu}
\date{\today}

\begin{abstract}
  We prove that operator algebras that have enough projections are completely determined by those projections, their symmetries,
  and the action of the latter on the former.  
  This includes all von Neumann algebras and all AW*-algebras.
  We introduce \emph{active lattices}, which are formed from these three
  ingredients. More generally, we prove that the category of AW*-algebras is
  equivalent to a full subcategory of active lattices. Crucial ingredients are an
  equivalence between the 
  category of piecewise AW*-algebras and that of piecewise complete
  Boolean algebras, and a refinement of the piecewise algebra
  structure of an AW*-algebra that enables recovering its total
  structure.  
\end{abstract}


\maketitle

\section{Introduction}
\label{sec:introduction}

Operator algebras play a major role in modern functional analysis and
mathematical physics, particularly algebras with an ample supply
of projections. Such algebras 
display a rich interplay between their algebraic structure, the order-theoretic
structure of their projections, the group-theoretic structure of their
unitaries, and their various topological structures. It is therefore
natural to wonder to what extent one of these aspects determines the
others. We will consider algebras for whom operator topologies play
a minor role, and focus on the
other facets; specifically, we work with AW*-algebras, which include all von Neumann algebras.
Such algebras are not completely determined by the
group-theoretic structure of their unitaries: for example, $U(A) \cong
U(A\op)$, but $A \not\cong A\op$ in general~\cite{connes:factor}.
Adding the order-theoretic structure of their projections does not
suffice to reconstruct the algebra either: again, $\Proj(A)
\cong \Proj(A\op)$.
Closely related to projections is the structure of the normal part
$N(A)$ of $A$ as a
\emph{piecewise\footnote{We prefer the terminology `piecewise algebra' 
over the traditional `partial algebra', because of the unfortunate
conjunction `partial complete Boolean algebra'.} algebra} (see~\cite{vdbergheunen:colim}). Roughly, these are
algebras where one can only add or multiply commuting elements. 
But adding this structure is still not enough to determine the
algebra, since $N(A)$ and $N(A\op)$ are isomorphic as piecewise algebras.
It follows from our main result that taking into account one final
ingredient does suffice to completely determine the algebra structure, namely the
action by conjugation of the unitaries on the projections. Thus we
answer the following preserver problem. 

\begin{corollary*}
  Let $A$ and $B$ be AW*-algebras. 
  If $f \colon N(A) \to N(B)$ is an isomorphism of piecewise algebras, that
  restricts to isomorphisms $\Proj(A) \cong \Proj(B)$ and 
  $U(A) \cong U(B)$, and satisfies
  $f(upu^*)=f(u)f(p)f(u)^*$, then $A \cong B$.
\end{corollary*}

There is considerable overkill in the previous corollary. 
For one thing, we could have stated the assumption on piecewise
algebras in terms that a priori contain less
information, such as, for example, the partial orders of commutative
subalgebras of $A$ and $B$ (see~\cite{hamhalter:jordan}), or various
notions built on those. 
We will prove that any isomorphism $\Proj(A) \to \Proj(B)$ extends
uniquely to an isomorphism $N(A) \to N(B)$ of piecewise algebras, that
could therefore have been left out from the assumptions altogether.
This puts the following two driving questions
on an equal footing. 
\begin{itemize}
\item What extra data make projections a complete
  invariant of AW*-algebras? 
\item What extra data on piecewise AW*-algebras enable
  extension to total ones?
\end{itemize} 
Moreover, it suffices to consider the subgroup of the
unitaries generated by so-called symmetries (see~\cite[Chapter~6]{alfsenshultz:statespaces}). 
Finally, the projection lattice injects into the symmetry group by $p
\mapsto 1-2p$, and so the projection lattice acts on itself in a 
certain sense. We can package up the remaining data in an \emph{active
lattice}, which therefore completely determines the AW*-algebra
structure. The precise definition can be found in
Section~\ref{sec:activelattices}, but let us emphasize here that it is
expressed exclusively in terms of the projection lattice and the
symmetry group that it generates. (For related ideas, see also~\cite{mayet:orthosymmetric}, which only came to our attention when the current work was already in press.)

In fact, we will be (quite) a bit more general, and work with
arbitrary morphisms instead of just isomorphisms: we define a functor
from the category of AW*-algebras to that of active lattices, and prove it 
to be full and faithful. 
This then implements our main result, which makes precise the titular claim that
an AW*-algebra is completely determined by its active lattice.

\begin{theorem*}
  The category of AW*-algebras is equivalent to a full subcategory of the category of active lattices.
\end{theorem*}

This is summarized in the following commuting diagram of
functors. Solid arrows represent functors that are faithful but not
full, whereas the dashed functor we construct is both
full and faithful. 
\[\xymatrix@R-2ex@C+2ex{
  & \AWstar \ar_-{\Proj}[dl] \ar^-{\Sym}[dr] \ar@{..>}[d] \\
  \cOML & \AL \ar[l] \ar[r] & \Group 
}\]
In particular, this result incorporates all von Neumann algebras, as W*-algebras and normal $*$-homomorphisms form a full subcategory of AW*-algebras.	

\subsubsection*{Motivation}

Our main motivation is to generalize the duality of commutative
C*-algebras and their Gelfand spectra to the noncommutative case.
Many proposals for noncommutative spectra have been studied. One of them concerns
\emph{quantales}~\cite{mulvey:andthen}, that are based on projection lattices
in the case of AW*-algebras. However, there are rigorous
obstructions to various categories being in duality with that of
C*-algebras, including that of quantales~\cite{reyes:obstructing,vdbergheunen:localicnogo}.
These obstructions suggest that 
a good notion of spectrum can instead be based on piecewise structures~\cite{heunenlandsmanspitterswolters:gelfand,hamhalter:jordan}. 
Our active lattices come very close to quantales, but circumvent the
obstruction afflicting them. Whereas a quantale is a monoid that
is also a lattice, an active lattice can be regarded as a
monoid that is generated by a lattice.
Stone duality between Stonean spaces and complete Boolean
algebras (see Section~\ref{sec:duality}) allows one to consider
$\Proj(A)$ as a substitute for the Gelfand spectrum in case $A$ is
a commutative AW*-algebra. So our results can also be regarded as a successful
extension of this ``substitute spectrum'' to noncommutative AW*-algebras.
This goal explains why go to the nontrivial trouble of take morphisms seriously, and deal with arbitrary morphisms with different domain and codomain rather than just focusing on isomorphisms.

The theorem above succeeds in extending a combination of Gelfand's and Stone's representation theorems noncommutatively for the case of AW*-algebras. Because of the relation to complete Boolean algebras sketched above, active lattices could be regarded as ``noncommutative Boolean algebras'', providing progress toward a category of ``noncommutative sets''.  This is an important step closer to the ``noncommutative topological spaces'' that C*-algebras represent than the ``noncommutative measure spaces'' of von Neumann algebras. This explains why we take pains to avoid measure-theoretic arguments and work with AW*-algebras instead of von Neumann algebras.

Our results can also be regarded as a novel answer to the Mackey--Gleason problem, 
that has been studied in great detail for von Neumann 
algebras. This type of problem asks what properties of a function
between projection lattices ensure that it extends to a linear
function between operator algebras, or more generally, what properties
of a function between operator algebras that is only piecewise linear
make it linear~\cite{buncewright:mackeygleason}. 
As mentioned, we generalize many constructions from von Neumann algebras to AW*-algebras, as the latter are the natural home for our arguments. In particular, we will not rely on Gleason's theorem to extend piecewise linearity to linearity~\cite{buncewright:jordan}, but directly generalize results to due to Dye~\cite{dye:projections} instead.
In addition, our main results hold perfectly well for algebras with $\mathrm{I}_2$
summands, which are exceptions to many classic theorems, including the Mackey--Gleason problem. Thus our main results answer this problem by approaching it a substantially different and worthwhile way.


\subsubsection*{Structure of the paper}

The article proceeds as follows. Section~\ref{sec:duality} recalls
AW*-algebras, complete Boolean algebras, and their piecewise
versions. It then proves that the two resulting categories of piecewise structures
are equivalent. Section~\ref{sec:activelattices} introduces
active lattices after discussing the ingredients of projection
lattices and symmetry groups. It also constructs the functor taking an
AW*-algebra to its active lattice. Section~\ref{sec:fullness} is
devoted to proving that this functor is full.

\section{(Piecewise) AW*-algebras and complete Boolean algebras}
\label{sec:duality}

After reviewing commutative AW*-algebras and their equivalence to
complete Boolean algebras, this section extends the equivalence to 
piecewise AW*-algebras and piecewise complete Boolean algebras,
positively answering~\cite[Remark~3]{vdbergheunen:colim}. 

\subsection*{AW*-algebras and complete Boolean algebras}

Kaplansky introduced AW*-algebras as an abstract generalization of von
Neumann algebras~\cite{kaplansky:awstar, berberian}. Their main characteristic is
that they are algebraically determined by their projections, \ie\ self-adjoint
idempotents, to a great extent. We denote the set of projections of a
$*$-ring $A$ by $\Proj(A)$. This set is partially ordered by the relation
$p \leq q \iff p = pq\, (= qp)$.

\begin{definition}\label{def:awstar}
  An \emph{AW*-algebra} is a C*-algebra $A$ that satisfies the
  following left-right symmetric and equivalent conditions:
  \begin{enumerate}[(a)]
  \item the right annihilator of any subset is generated as
    right ideal by a projection;
  \item the right annihilator of any element $a \in A$ is
    generated by a projection, and $\Proj(A)$ forms a
    complete lattice;
  \item the right annihilator of any element $a \in A$ is generated
    by a projection, and every orthogonal family in
    $\Proj(A)$ has a supremum;
  \item any maximal commutative subalgebra $C$ is the closed linear
    span of $\Proj(C)$, and every orthogonal family in $\Proj(A)$ has
    a supremum.
  \end{enumerate}
  A \emph{morphism of AW*-algebras} is a $*$-homomorphism that preserves
  suprema of projections. We write $\AWstar$ for the category of
  AW*-algebras and their morphisms.
\end{definition}

For a subset $S$ of an AW*-algebra $A$, write $\Rann(S)$ for
the (unique) projection of Definition~\ref{def:awstar}(a): $\Rann(S)$ is
the least projection annihilating every element of $S$, and is also the
(unique) projection such that $xy = 0$ for all $x \in S$ if and
only if $\Rann(S)y = y$. This is the \emph{right annihilating projection}
of $S$. With a slight abuse of notation, we write $\Rann(a)$ in place
of $\Rann(\{a\})$ for a single element $a \in A$. The projection
$\RP(a) = 1 - \Rann(a)$ is the \emph{right supporting projection}
of $a$. It is the least projection satisfying $a \RP(a) = a$.

Given the equivalent conditions defining AW*-algebras, there are
several possible choices for morphisms of the category \AWstar.
Fortunately, the most obvious conditions one might impose on
a $*$-homomorphism are also equivalent, as the following lemma shows.
Recall that a set of projections is called directed when every pair
of its elements has an upper bound within the set.

\begin{lemma}\label{lem:awstarmorphisms}
  For a $*$-homomorphism $f \colon A \to B$ between AW*-algebras, the
  following conditions are equivalent:
  \begin{enumerate}[\quad (a)]
  \item $f$ preserves right annihilating projections of arbitrary subsets;
  \item $f$ preserves suprema of arbitrary families of projections;
  \item $f$ preserves suprema of orthogonal families of projections;
  \item $f$ preserves suprema of directed families of projections.
  \end{enumerate}
  If $f$ satisfies these equivalent conditions, then the kernel of $f$
  is generated by a central projection and $f$ preserves $\RP$.
\end{lemma}
\begin{proof}
  For a morphism $f$ satisfying (c), the last sentence of the lemma
  follows from~\cite[Exercise~23.8]{berberian}. That (b)
  implies (a) follows from the fact that such $f$ preserves $\RP$ as
  well as the following equation for any $S \subseteq A$,
  \[
     \Rann(S) = \bigvee_{x \in S} \Rann(x) = \bigvee_{x \in S} (1 - \RP(x))
  \]
  (see also~\cite[Proposition~4.2]{berberian}). Conversely,
  assume (a), and let $\{p_i\} \subseteq \Proj(A)$. We will prove
  that
  \[
      \bigvee \{p_i\} = \Rann(\{1-p_i\}),
  \]
  from which (a) $\Rightarrow$ (b) will follow. Writing $p =
  \Rann(\{1-p_i\})$, every $(1-p_i) \perp p$, which gives $p_i \leq p$
  for all $i$. And if all $p_i  \leq q$ for any $q \in \Proj(A)$, then
  all $(1-p_i) \perp q$, which means that $pq = q$ and thus $q \leq p$.
  Hence $p = \bigvee_i p_i$ as desired.

  Clearly (b) $\Rightarrow$ (d). To see (d) $\Rightarrow$ (c), let
  $\mathcal{P}$ be an orthogonal family of projections. Setting $q_S =
  \bigvee S$ for every finite subset $S \subseteq \mathcal{P}$ gives a
  directed family of projections with the same supremum as
  $\mathcal{P}$. Because each $S$ is orthogonal and finite, we have
  $f(q_S) = f(\sum S) = \sum f(S) = \bigvee f(S)$. And because $f$ is
  assumed to preserve directed suprema, $f(\bigvee \mathcal{P}) = f(\bigvee_S q_S) =
  \bigvee_S f(q_S) = \bigvee_S f(S) = \bigvee f(\mathcal{P})$.

  Finally, because any
  $*$-homomorphism $f \colon A \to B$ between AW*-algebras 
  restricts to a lattice homomorphism $\Proj(A) \to
  \Proj(B)$ (see~\cite[Proposition~5.7]{berberian}),
  Lemma~\ref{lem:orthocomplete} below provides a direct proof of (c)
  $\Rightarrow$ (b). 
\end{proof}

Observe that the proof of the previous lemma establishes more than was
promised: it shows that direct sums provide finite products in the category
\AWstar. 
The initial object is the AW*-algebra $\C$, and the terminal object is the
zero algebra.
Observe also that the above lemma holds true if $f$ is only assumed
to be a $*$-ring homomorphism. This will be useful later in
Section~\ref{sec:fullness}. 

Let $\Wstar$ denote the category of W*-algebras
(\ie\ abstract von Neumann algebras) and 
normal $*$-homomorphisms. Then $\Wstar$ is a full subcategory of
$\AWstar$. (The objects of $\Wstar$ are objects of $\AWstar$
by~\cite[Proposition~4.9]{berberian}, and the subcategory can
be shown to be full, for instance, by composing a $*$-homomorphism $A \to B$
with all normal linear functionals on $B$ and
using~\cite[Corollary~III.3.11]{takesaki:operatoralgebras1}. See
also~\cite[Lecture~11]{lurie:neumann}.)
In particular, the lemma above provides equivalent conditions for
a $*$-homomorphism between von Neumann algebras to be
normal.

If an AW*-algebra is commutative, its projections form a
\emph{complete Boolean algebra}: a distributive lattice in which
every subset has a least upper bound, and in which every element has a
complement. In fact, we now detail an equivalence between the
categories of commutative AW*-algebras and complete Boolean algebras.

First recall Stone duality~\cite[Corollary~II.4.4]{johnstone:stonespaces},
which gives a dual equivalence between Boolean algebras and Stone
spaces, \ie\ totally disconnected compact Hausdorff spaces. If the Boolean
algebra is complete, the corresponding Stone space is in fact a Stonean
space, \ie\ extremally disconnected, meaning that the closure of every open
set is (cl)open. We write $\CBoolean$ for the category of complete
Boolean algebras and homomorphisms of Boolean algebras that preserve
arbitrary suprema.
On the topological side, we write \cat{Stonean} for 
the category of Stonean spaces and open continuous functions. With
this choice of morphisms, Stone duality restricts to a dual
equivalence between $\CBoolean$ and
$\cat{Stonean}$. See~\cite[Section~6]{bezhanishvili:devries}.

Similarly, recall that Gelfand duality gives a dual equivalence
between commutative C*-algebras and compact Hausdorff spaces. If the
C*-algebra is an AW*-algebra, then the compact Hausdorff space is in
fact a Stonean space~\cite[Theorem~7.1]{berberian}.
If we write \cat{cAWstar} for the full subcategory of
\cat{AWstar} consisting of commutative AW*-algebras, then Gelfand
duality restricts to a dual equivalence between \cat{cAWstar} and
\cat{Stonean}. 
Hence we have the following equivalences.
\begin{equation}\label{eq:commutativeduality}
  \xymatrix@C+4ex{
    \cat{cAWstar} \ar@{}|-{\simeq}[r] \ar@<1ex>^-{\Spec}[r]
    & \cat{Stonean}\op \ar@{}|-{\simeq}[r] \ar@<1ex>^-{\Clopen}[r] \ar@<1ex>^-{\Cont}[l]
    & \CBoolean \ar@<1ex>^-{\Stone}[l]
  }
\end{equation}

Explicitly, $\Spec$ is the functor taking characters and furnishing
them with the Gelfand topology, and $\Clopen$ takes clopen subsets,
so the composite $\Clopen \circ \Spec$ is naturally isomorphic to the
functor $\Proj$. We write $\Func$ for the composite $\Cont \circ
\Stone$. Explicitly, $\Stone = \CBoolean(-,2)$ and $\Cont = C(-,\mathbb{C})$.  
Thus $\Proj$ and $\Func$ form an equivalence between
commutative AW*-algebras and complete Boolean algebras.

\subsection*{Piecewise structures}

Piecewise algebras are sets of which only certain pieces carry
algebraic structure, but in a coherent way. Before we can extend the
equivalence above to a piecewise setting, we spell out the
appropriate definitions. Definition~\ref{def:awstar}(c) leads to a
specialization of the definition of a piecewise C*-algebra, that we 
recall first~\cite{vdbergheunen:colim}. 

\begin{definition}
\label{def:pawstar}
  A \emph{piecewise C*-algebra} consists of a set $A$ with:
  \begin{itemize}
    \item a reflexive and symmetric binary
      (\emph{commeasurability}) relation $\commeas \subseteq A \times A$;
    \item elements $0,1 \in A$;
    \item a (total) involution $* \colon A \to A$;
    \item a (total) function $\cdot \colon \mathbb{C} \times A \to A$;
    \item a (total) function $\|\! - \!\| \colon A \to \mathbb{R}$;
    \item (partial) binary operations $+, \cdot \colon  \commeas \to A$;
  \end{itemize}
  such that every set $S \subseteq A$ of pairwise commeasurable
  elements is contained in a set $T \subseteq A$ of pairwise
  commeasurable elements that forms a commutative C*-algebra under the
  above operations.

  A \emph{piecewise AW*-algebra} is a piecewise C*-algebra $A$ with
  \begin{itemize}
    \item a (total) function $\RP \colon A \to \Proj(A)$;
    \item a (partial) operation $\bigvee \colon \{ X \subseteq \Proj(A) \mid X
      \times X \subseteq \commeas \} \to \Proj(A)$;
  \end{itemize}
  such that every set $S \subseteq A$ of pairwise commeasurable
  elements is contained in a set $T \subseteq A$ of pairwise
  commeasurable elements that forms a commutative AW*-algebra under
  the above operations.

  A \emph{morphism of piecewise AW*-algebras} is a (total) function $f
  \colon A \to B$ such that:
  \begin{itemize}
    \item $f(a) \commeas f(b)$ for commeasurable $a,b \in A$;
    \item $f(ab)=f(a)f(b)$ for commeasurable $a,b \in A$;
    \item $f(a+b)=f(a)+f(b)$ for commeasurable $a,b \in A$;
   \item $f(za) = zf(a)$ for $z \in \mathbb{C}$ and $a \in A$;
    \item $f(a)^* = f(a^*)$ for $a \in A$;
    \item $f(\bigvee_i p_i) = \bigvee_i f(p_i)$ for pairwise
        commeasurable projections $\{p_i\}$.
  \end{itemize}
  By Lemma~\ref{lem:awstarmorphisms}, such a morphism automatically
  satisfies $f(\RP(a)) = \RP(f(a))$. Also, it follows from the last
  condition that $f(1)=1$.
  Piecewise AW*-algebras and their morphisms organize themselves into
  a category denoted by $\cat{PAWstar}$.
\end{definition}

The prime example of a piecewise AW*-algebra is the set $N(A)$ of
normal elements of an AW*-algebra $A$, where commeasurability is given
by commutativity. Hence one can regard piecewise AW*-algebras as
AW*-algebras of which the algebraic structure between noncommuting
elements is forgotten.

\begin{lemma}\label{lem:normalfunctor}
  The assignment sending an AW*-algebra $A$ to its set of normal elements
  $N(A)$ defines a functor $N \colon \AWstar \to \PAWstar$.
\end{lemma}
\begin{proof}
  Let $A$ be an AW*-algebra. The natural piecewise algebra structure on $N(A)$
  is a piecewise $C^*$-algebra by~\cite[Proposition~3]{vdbergheunen:colim}.
  It is a piecewise AW*-algebra under the inherited $\RP$ 
  and supremum operations, because every pairwise commuting subset of
  $N(A)$ is contained in a maximal commutative subalgebra of $A$, that
  is an AW*-subalgebra by Definition~\ref{def:awstar}(d), and must
  itself necessarily be contained in $N(A)$.  
  Functoriality of $N$ is easy to check.
\end{proof}

The next lemma observes that the structures 
$\bigvee$ and $\RP$ in Definition~\ref{def:pawstar} are
in fact properties. (Nonetheless morphisms in $\PAWstar$ have
to preserve $\bigvee$.) Thus we may say that a certain piecewise C*-algebra
``is a piecewise AW*-algebra'' without ambiguity of the
AW*-operations.
We call a projection $p$ of a piecewise C*-algebra a \emph{least upper
commeasurable bound} of a commeasurable set $S$ of projections when $p
\commeas a$ for any $a$ that makes $S \cup \{a\}$ 
commeasurable, and whenever a projection $q$ is 
commeasurable with $S \leq q$, then $q$ is commeasurable with $p$
as well and $p \leq q$. 

\begin{lemma}
  Let $A$ be a piecewise C*-algebra. There is at most one choice of
  operations $\bigvee$ and $\RP$ as in Definition~\ref{def:pawstar} 
  making $A$ a piecewise AW*-algebra.
\end{lemma}
\begin{proof}
  First, we claim that in any piecewise AW*-algebra, $\bigvee$ is
  characterized as giving the least upper commeasurable bound.
  Since these rely only the underlying
  piecewise C*-algebra structure, $\bigvee$ is then unique. The claim
  derives from 
  Definition~\ref{def:pawstar} as follows. Since $\bigvee$ makes $A$
  into a piecewise AW*-algebra, there exists a commutative AW*-algebra
  $T \subseteq A$ whose suprema are given by $\bigvee$, containing
  $S$. Hence $T$ contains $\bigvee S$, making $S \cup 
  \{ \bigvee S \}$ commeasurable, and $\bigvee S$ majorizes $S$. If $q$ is
  commeasurable with $S$ and majorizes it, there exists an AW*-algebra
  $T$ containing $S \cup \{q\}$. In particular, it is closed under suprema
  of projections, which are given by $\bigvee$. Thus it contains
  $\bigvee S$, which is therefore commeasurable with $q$ and
  $\bigvee S \leq q$.
  Finally, $\RP(a) = \bigwedge \{ p \in \Proj(A)
  \mid ap=a \}$ equals $\bigvee \{ q \in \Proj(A) \mid \forall p \in
  \Proj(A) \colon ap=a \Rightarrow p \leq q\}$.
\end{proof}

The next two results give convenient ways to recognize 
piecewise AW*-algebras among piecewise C*-algebras. 
The first shows that a piecewise AW*-algebra is a piecewise
C*-algebra that is ``covered'' by sufficiently many AW*-algebras; 
recall that an AW*-algebra $A$ is an \emph{AW*-subalgebra} of an
AW*-algebra $B$ when the inclusion $A \hookrightarrow B$ is a morphism
in $\AWstar$.
The second is a characterization analogous to Kaplansky's
original definition of AW*-algebras as 
C*-algebras with extra properties, Definition~\ref{def:awstar}(d).

\begin{lemma}\label{RPandsup}
  A piecewise C*-algebra $A$ is a piecewise AW*-algebra when:
  \begin{itemize}
  \item any commeasurable subset $S$ is contained in a
    commeasurable subset $T(S)$ that is an AW*-algebra, such that:
  \item if $S \subseteq S'$ are commeasurable subsets, 
    $T(S)$ is an AW*-subalgebra of $T(S')$.
  \end{itemize}
\end{lemma}
\begin{proof}
  Define functions $\RP$ and $\bigvee$ by calculating $\RP(a)$ as in $T(\{a\})$,
  and calculating $\bigvee X$ as in $T(X)$. By~\cite[Proposition~3.8]{berberian},
  then $\RP(a)$ is the same when calculated in any $T(S)$ with
  $a \in S$, because $T(\{a\})$ is an AW*-subalgebra of $T(S)$.
  Similarly, $\bigvee{X}$ is the same in any $T(S)$ with $X \subseteq
  S$~\cite[Proposition~4.8]{berberian}. Therefore $\RP$ and $\bigvee$
  make $A$ into a piecewise AW*-algebra.
\end{proof}

\begin{proposition}
  A piecewise C*-algebra $A$ is a piecewise AW*-algebra when both:
  \begin{itemize}
  \item commeasurable sets of projection have least upper
    commeasurable bounds; 
  \item maximal commeasurable subalgebras are closed linear
    spans of projections.   
  \end{itemize}
\end{proposition}
\begin{proof}
  The first assumption defines a function $\bigvee$. 
  If $S$ is a commeasurable subset, Zorn's lemma provides
  a maximal commeasurable set $M \supseteq S$. By definition of
  piecewise C*-algebra, $M$ is contained in a commeasurable
  C*-algebra. Hence maximality guarantees that $M$ is a commutative
  C*-algebra under the operations of $A$. But now the second assumption
  together with $\bigvee$ make $M$ into an
  AW*-algebra~\cite[Exercise~7.1]{berberian}. Taking 
  $S=\{a\}$, we can define $\RP(a)$ as the unique right supporting
  projection in $M$.
  The functions $\bigvee$ and $\RP$ (uniquely) make $A$ into a piecewise
  AW*-algebra.  
\end{proof}

 There is a similar definition of piecewise
complete Boolean algebras that specializes the definition of piecewise
Boolean algebras~\cite{vdbergheunen:colim}.

\begin{definition}
  A \emph{piecewise complete Boolean algebra} consists of a set $B$ with
  \begin{itemize}
    \item a reflexive and symmetric binary
      (\emph{commeasurability}) relation $\commeas \subseteq B \times B$;
    \item a (total) unary operation $\lnot \colon B \to B$;
    \item a (partial) operation $\bigvee \colon \{ X \subseteq B \mid
      X \times X \subseteq \commeas \} \to B$;
  \end{itemize}
  such that every set $S \subseteq B$ of pairwise commeasurable
  elements is contained in a pairwise commeasurable set $T \subseteq
  B$ that forms a complete Boolean algebra under 
  the above operations.
  (Notice that these data uniquely determine elements $0=\bigvee
    \emptyset$ and $1=\neg 0$, and (partial) operations $x \vee y = \bigvee
    \{x,y\}$ and $x \wedge y = 
    \neg(\neg x \vee \neg y)$.)

  A \emph{morphism of piecewise complete Boolean algebras} is a
  (total) function that preserves commeasurability and all the
  algebraic structure, whenever defined. We write $\PCBoolean$ for the
  resulting category. 
\end{definition}

\subsection*{A piecewise equivalence}
 
The functor $\Proj \colon \AWstar \to \CBoolean$ extends to
a functor $\PAWstar \to
\PCBoolean$~\cite[Lemma~3]{vdbergheunen:colim}. We aim to prove 
that the latter functor is also (part of) an equivalence. 
By~\cite[Theorem~3]{vdbergheunen:colim}, any piecewise complete Boolean
algebra $B$ can be seen as (a colimit of) a functor $\cC(B) \to
\CBoolean$, where $\cC(B)$ is the diagram of (commeasurable)
complete Boolean subalgebras of $B$ and inclusions. Similarly, by the
AW*-variation of~\cite[Theorem~7]{vdbergheunen:colim}, any 
piecewise AW*-algebra $A$ can be seen as a functor $\cC(A) \to
\cAWstar$, where $\cC(A)$ is the diagram of (commeasurable)
commutative AW*-subalgebras of $A$ and inclusions. Hence
postcomposition with $\Func$ should turn a 
piecewise complete Boolean algebra into a piecewise AW*-algebra. Below we
explicitly compute the ensuing colimit to get a functor $F \colon
\PCBoolean \to \PAWstar$. Even though it is unclear how
general coequalizers are computed in either category, the fact that
$\cC(B)$ is a diagram of monomorphisms makes the constructions manageable. 

\begin{lemma}\label{monics}
  The monomorphisms in $\AWstar$, $\cAWstar$, and $\CBoolean$
  are precisely the injective morphisms. 
\end{lemma}
\begin{proof}
  Let $f \colon A \rightarrowtail B$ be a monomorphism in
  $\AWstar$ or $\cAWstar$. We first show that $\Proj(f) \colon \Proj(A)
  \rightarrowtail \Proj(B)$ is injective. Suppose that $f(p)=f(q)$ for
  $p,q \in \Proj(A)$. Define $g,h \colon \mathbb{C}^2 \to A$ by
  $g(1,0)=p$ and $h(1,0)=q$. Then $(f \circ g)(x,y) =
  xf(p)+yf(p)^\perp = (f \circ h)(x,y)$, so $g=h$ and hence $p=q$. In
  particular, $f$ cannot map a nonzero projection of $A$ to 0 in
  $B$. Thus $\ker(f)=0$ by Lemma~\ref{lem:awstarmorphisms}, and $f$ is
  injective. Conversely, injective morphisms are trivially monic.

  Monomorphisms $f \colon P \rightarrowtail Q$ in $\CBoolean$
  factor as 
  \[
    P \cong \Proj(\Func(P)) \rightarrowtail \Proj(\Func(Q)) \cong Q.
  \]
  Now, isomorphisms in $\CBoolean$ are bijective, and
  by the above, the middle arrow $\Proj(\Func(f))$ is injective,
  making $f$ itself injective. 
\end{proof}

We are ready to define the object part of a functor $F \colon
\PCBoolean \to \PAWstar$.

\begin{definition}\label{defFB}
  Let $B$ be a piecewise complete Boolean algebra. Define $F(B)$ to be
  the following collection of data.
  \begin{itemize}
  \item The carrier set $A$ is $(\coprod_{C \in \cC(B)}
    \Func(C)) /\sim$, where $\sim$ is the smallest equivalence relation
    satisfying $f \sim g$ for $f \in \Func(C)$ and $g \in \Func(D)$
    when $C \subseteq D$ and $g=\Func(C \hookrightarrow D)(f)$.
  \item Two equivalence classes $\rho$ and $\sigma$ in $A$ are
    commeasurable if and only if there exist $C \in \cC(B)$ and $f,g \in
    \Func(C)$ such that $f \in \rho$ and $g \in \sigma$.
  \item Notice that $z \cdot 1_C \sim z \cdot 1_D$ for $C \subseteq D$
    in $\cC(B)$, and any $z \in \mathbb{C}$. Also, $\{0,1\}$ is the
    minimal element of $\cC(B)$. Hence $z \cdot 1_C \sim z \cdot 1_D$
    for any $C,D \in \cC(B)$ by transitivity.

    In particular, $[0_{\Func(\{0,1\})}]=[0_C]$ defines an element $0
    \in A$ independently of $C$, and $1 \in A$ is defined by
    $[1_{\Func(\{0,1\})}]=[1_C]$ for any $C \in \cC(B)$. Likewise,
    $z \cdot [f] = [z \cdot f]$ is well-defined for $z \in
    \mathbb{C}$.
  \item Similarly, $[f]^* = [f^*]$ gives a well-defined operation $* \colon A \to A$. 
  \item If $\rho$ and $\sigma$ are two commeasurable elements of $A$,
    then by definition there are $C \in \cC(B)$ and $f,g \in \Func(C)$
    with $f \in \rho$ and $g \in \sigma$. Setting $\rho + \sigma =
    [f+g]$ and $\rho \cdot \sigma = [f \cdot g]$ gives well-defined
    operations $+,\cdot \colon \commeas \to A$.
  \item If $C \subseteq D$, then $\Func(C \hookrightarrow D) \colon \Func(C) \to
    \Func(D)$ is an injective $*$-homomorphism by Lemma~\ref{monics},
    and hence preserves norm~\cite[Theorems~4.1.8,~4.1.9]{kadisonringrose}. 
    So $\|[f]\| = \|f\|$ gives a well-defined operation $A \to
    \mathbb{R}$.
\end{itemize}
\end{definition}

\begin{proposition}
  The data $F(B)$ defined above form a piecewise AW*-algebra.  
\end{proposition}
\begin{proof}
  For $\rho$ in $A$, define
  $
    C_\rho = \bigcap \{ C \in \cC(B) \mid \rho \cap \Func(C) \neq
    \emptyset \}.
  $
  Because $\cC(B)$ is closed under arbitrary intersections, $C_\rho \in
  \cC(B)$. If $\rho$ and $\sigma$ in $A$ are commeasurable, then by
  definition there are $C \in \cC(B)$ and $f,g \in \Func(C)$, so
  $C_\rho \subseteq C \supseteq C_\sigma$. But that implies any
  element of $C_\rho$ is commeasurable in $B$ with any element of $C_\sigma$.

  Let $S \subseteq A$ be pairwise commeasurable. Then $\hat{S} =
  \bigcup_{\rho \in S} C_\rho \subseteq B$ is pairwise
  commeasurable by the last paragraph. Hence there exists a set $\hat{T} \subseteq B$ that
  contains $\hat{S}$, is pairwise commeasurable, and forms a complete
  Boolean algebra under the operations from $B$. Therefore
  $T=\{ [f] \mid f \in \Func(\hat{T}) \} \subseteq A$ contains $S$, is
  commeasurable, and forms a commutative AW*-algebra under the
  operations from $A$.
  Hence $A$ is a piecewise C*-algebra. Moreover, if $S \subseteq S'$,
  then $\hat{S} \subseteq \hat{S'}$, and $\hat{T} \subseteq
  \hat{T'}$ are both complete Boolean subalgebras of $B$ under the
  same operation $\bigvee$, namely that of $B$. Hence $T$ is an
  AW*-subalgebra of $T'$, so that $A$ is in fact a piecewise
  AW*-algebra by Lemma~\ref{RPandsup}.
\end{proof}

\begin{lemma}\label{Fcolim}
  If $B \in \PCBoolean$, then $F(B)$ is a colimit of the diagram
  $\Func(C)$ with $C$ ranging over $\cC(B)$. Therefore $F$ is
  functorial $\PCBoolean \to \PAWstar$. 
\end{lemma}
\begin{proof}
  Clearly there exists a cocone of morphisms $\Func(C) \to A$ for
  each $C \in \cC(B)$, given by $f \mapsto [f]$.
  If $k_C \colon \Func(C) \to A'$ is another cocone, the unique
  mediating map $m \colon A \to A'$ is given by $m([f])=k_C(f)$ when
  $f \in \Func(C)$.

  Let $g \colon B_1 \to B_2$ be a morphism of $\PCBoolean$.
  Because $F(B_1)$ is a colimit of $\{\Func(C) \mid C \in \cC(B_1)\}$,
  to define a morphism $F(g) \colon F(B_1) \to F(B_2)$, it suffices to
  specify morphisms $\Func(C) \to F(B_2)$ in $\PAWstar$ for each
  $C \in \cC(B_1)$. But $g$ preserves commeasurability, so its restriction to
  $C$ is a morphism in $\CBoolean$ and we can just take
  $F(g)\big|_{\Func(C)} = \Func(g\big|_C)$. This assignment is automatically functorial.
  Moreover, it is well-defined, even though colimits are only unique
  up to isomorphism, because Definition~\ref{defFB} fixed one specific
  colimit.
\end{proof}

\begin{theorem}\label{thm:piecewiseequivalence}
  The functors $F$ and $\Proj$ form an equivalence between the categories
  $\PAWstar$ and $\PCBoolean$.
\end{theorem}
\begin{proof}
  For a piecewise AW*-algebra $A$ we have
  \begin{align*}
    F(\Proj(A))
    & \cong \colim_{C \in \cC(\Proj(A))} \Func(C)  \\
    & \cong \colim_{C \in \cC(A)} \Func(\Proj(C)) \\
    & \cong \colim_{C \in \cC(A)} C 
       \cong A.
  \end{align*}
  by Lemma~\ref{Fcolim}, \cite[Proposition~6]{vdbergheunen:colim}, and
  \cite[Theorem~7]{vdbergheunen:colim}. 
  Each of the above isomorphisms is readily seen to be natural in $A$.

  Next we establish an isomorphism $\Proj(F(B)) \cong B$. 
  Let $\rho \in \Proj(F(B)) \subseteq F(B)$. If $\rho = [f]$ for $f \in
  \Func(C)$ and $C \in \cC(B)$, then $f \in \Proj(\Func(C))$. So $\eta_C(f) \in C
  \subseteq B$, where $\eta$ is the unit of the equivalence formed by
  $\Proj$ and $\Func$. In fact, by naturality of $\eta$, if $C \subseteq D$
  for another $D \in \cC(B)$ with the inclusion denoted by $i \colon C
  \hookrightarrow D$, the following diagram commutes.
  \[\xymatrix@R-2ex{
    & \Proj(\Func(C)) \ar^-{\eta_C}[r] \ar_-{\Proj(\Func(i))}[d]
    & C \ar@{^{(}->}^-{i}[d] \\
    & \Proj(\Func(D)) \ar_-{\eta_D}[r]
    & D
  }\]
  So if $g \sim f$ because $g=\Func(C \hookrightarrow D)(f)$, then
  $\eta_C(f)=\eta_D(g)$. That is, $\eta_C(f)$ is independent of the
  chosen representative $f$ of $\rho$. Thus we have a map $\Proj(F(B))
  \to B$ that is a morphism of piecewise complete Boolean algebras,
  because $\eta$ is a morphism of complete Boolean algebras.

  Conversely, for $b \in B$, consider the commeasurable subalgebra
  $\generated{B}{b}$ of $B$ generated by $b$.
  Then $\eta^{-1}_{\generated{B}{b}}(b)$ is an
  element of $\Proj(\Func(\generated{B}{b}))$. Thus
  $b \mapsto [\eta^{-1}_{\generated{B}{b}}(b)]$ is a function $B \to
  \Proj(F(B))$, that is easily seen to be inverse to the function
  $\Proj(F(B)) \to B$ above. Thus we have an isomorphism $B \cong
  \Proj(F(B))$ of piecewise complete Boolean algebras. Unfolding
  definitions shows that this isomorphism is natural in $B$.
\end{proof}

As a consequence, the functor $\Proj$ preserves general coequalizers.

\section{The category of active lattices}
\label{sec:activelattices}

This section equips the piecewise AW*-algebra structure of
AW*-algebras $A$ with enough extra data to recover their full algebra
structure, which will be done in the next section. The required
structure consists of three ingredients: a lattice structure on
$\Proj(A)$, a group structure on the so-called symmetry subgroup of
the unitaries $U(A)$, and an action of the latter on the former. We
will discuss each in turn.

\subsection*{The projection lattice}

We start with some axioms satisfied by lattices of projections of AW*-algebras.

\begin{definition}
  An \emph{orthocomplementation} on a lattice $P$ is an
  order-reversing involution $p \mapsto p^{\perp}$ satisfying 
  $p \vee p^{\perp} = 1$ and $p \wedge p^{\perp} = 0$ (\ie, $p^{\perp}$
  is a \emph{complement} of $p$). We say $p$ and $q$
  are \emph{orthogonal} when $p \leq q^\perp$. 
  An orthocomplemented lattice is said to be
  \emph{orthomodular} when
  $p \vee (p^{\perp} \wedge q) = q$ for all $p \leq q$.
  Complete orthomodular lattices form a category $\cOML$ 
  whose morphisms are functions that preserve the
  orthocomplementation as well as arbitrary suprema. 
\end{definition}

The condition of being an object of $\cOML$
can be tested on orthogonal subsets, and the same is
nearly true for morphisms. 

\begin{lemma}\label{lem:orthocomplete}
  An orthomodular lattice $P$ is complete if and only if every 
  orthogonal subset of $P$ has a least upper bound. If $P$ and $Q$
  are complete orthomodular lattices, a function $f \colon P \to Q$ is
  a morphism of $\cOML$ if and only if it preserves orthocomplements,
  binary joins, and suprema of orthogonal sets. 
\end{lemma}
\begin{proof}
  The first statement is~\cite[Corollary~1]{holland:orthocomplete}.
  Let $f \colon P \to Q$ be as in the second statement, and let
  $\{p_i\}$ be any subset of $P$. Because $f$ preserves finite joins,
  it preserves order, and so $\bigvee f(p_i) \leq f(\bigvee p_i)$; we
  prove the reverse comparison. 
  Let $\{e_\alpha\}$ be a maximal orthogonal set of nonzero elements of $P$
  with $f(e_\alpha) \leq \bigvee f(p_i)$, and set $e = \bigvee e_\alpha$.
  By hypothesis, $f(e) = \bigvee f(e_\alpha) \leq \bigvee f(p_i)$.
  Thus it suffices to show that each $p_i \leq e$, for then $\bigvee p_i \leq e$
  and $f(\bigvee p_i) \leq f(e) \leq \bigvee f(p_i)$ as desired. Assume for
  contradiction that some $p_j \nleq e$. Then $e' = (p_j \vee e) \wedge e^\perp$
  is a nonzero element of $P$ orthogonal to $e$ and hence orthogonal
  to each $e_\alpha$. Furthermore
  \[
  f(e') \leq f(p_j \vee e) = f(p_j) \vee f(e) \leq \bigvee f(p_i),
  \]
  since $e' \leq p_j \vee e$. But this contradicts the maximality of $\{e_\alpha\}$.
\end{proof}

The axioms defining AW*-algebras and their morphisms are such that the
operation of passing to projection lattices defines a functor $\Proj
\colon \AWstar \to \cOML$. 

Complete orthomodular lattices are tightly linked to piecewise
complete Boolean algebras (rather than the more general
orthocomplemented lattices).
Indeed, any complete orthomodular lattice $P$ canonically is a
piecewise complete Boolean algebra, as follows. Define a
commeasurability relation  
$\odot$ on $P$ by the following equivalent
conditions, for any $p, q \in P$: 
\begin{enumerate}[\quad (i)]
\item there is a Boolean subalgebra of $P$ that contains both
  $p$ and $q$;
\item there exist pairwise orthogonal $p', q', r \in P$ with
  $p = p' \vee r$ and $q = q' \vee r$;
\item $p \wedge (p \wedge q)^\perp$ is orthogonal to $q$;
\item $q \wedge (p \wedge q)^\perp$ is orthogonal to $p$;
\item the \emph{commutator} 
  $(p \vee q) \wedge (p \vee q^\perp) \wedge (p^\perp \vee q)
  \wedge (p^\perp \vee q^\perp)$ of $p$ and $q$ is zero.
\end{enumerate}
For the equivalence of (i)--(iv) we refer
to~\cite[Lemma~6.7]{varadarajan:geometry1}; for the equivalence of
(i) and (v) see~\cite{marsden:commutator}.

\begin{lemma}\label{lem:piecewiseorthomodular}
  The assignment $P \mapsto (P,\odot)$ is a functor
  $\cOML \to \PCBoolean$.
\end{lemma}
\begin{proof}
  Given a complete orthomodular lattice $P$ and the commeasurability
  relation $\odot$ above, it follows from~\cite[Lemma~6.10]{varadarajan:geometry1}
  that the supremum operation of $P$ restricts to a partial operation
  $\bigvee \colon \{X \subseteq P \mid X \times X \subseteq \odot\} \to P$.
\end{proof}

Composing this forgetful functor with the equivalence $\PCBoolean \to
\PAWstar$ of Theorem~\ref{thm:piecewiseequivalence} gives a
canonical functor $\cOML \to \PAWstar$.
Below, we will extend the structure of the piecewise complete
Boolean algebra $\Proj(A)$ to that of a complete orthomodular lattice,
where $A$ is a piecewise AW*-algebra. As a converse to the above
lemma, we now show that this is a property rather than structure.

For any piecewise Boolean algebra $B$, let $\leq$ be the union of the
partial orders on each commeasurable subalgebra $C$ of $B$.
When this relation is transitive, it is a partial order, which we call
the induced partial order. In that case we call $B$ \emph{transitive}.
If every pair of (not necessarily commeasurable) elements of $B$
have a least upper bound with respect to $\leq$, we say that $B$
is \emph{joined}.
Similarly, we call a piecewise AW*-algebra $A$ transitive or joined
when $\Proj(A)$ is respectively transitive or joined.

\begin{proposition}
  The following categories are equivalent:
  \begin{enumerate}[\quad (a)]
  \item the category $\cOML$ of complete orthomodular lattices;
  \item the subcategory of $\PCBoolean$ whose objects are transitive and joined
    and whose morphisms preserve binary joins;
  \item the subcategory of $\PAWstar$ whose objects are transitive and joined
    and whose morphisms preserve binary joins of projections.
  \end{enumerate}
\end{proposition}
\begin{proof}
  The piecewise complete Boolean algebras that are in the image of the
  functor $\cOML \to \PCBoolean$ from
  Lemma~\ref{lem:piecewiseorthomodular} are by definition transitive and joined.  
  Next, we define a functor $G$ in the opposite direction. Let $B$ be a transitive,
  joined piecewise complete Boolean algebra and $\leq$ its induced partial
  order. By construction of $\leq$, it restricts to the given partial order on
  each commeasurable subalgebra of
  $B$. Furthermore, it is straightforward to verify that if $X
  \subseteq B$ is commeasurable then $\bigvee X$ is the least upper
  bound of $X$ with respect to $\leq$. Kalmbach's bundle
  lemma~\cite[1.4.22]{kalmbach:orthomodularlattices} now applies to 
  show that $\leq$ and $\neg$ induce the structure of an orthomodular
  lattice on $B$. Because orthogonal subsets are commeasurable, and
  $B$ has suprema of such subsets, it in fact has suprema of
  arbitrary subsets by Lemma~\ref{lem:orthocomplete}. 
  This makes $B$ into a complete orthomodular lattice, and we can define
  $G(B)=(B,\leq)$. Setting $G(f)=f$ on for $\PCBoolean$ morphisms
  that preserve binary joins gives a well-defined
  functor, thanks to Lemma~\ref{lem:orthocomplete}.
  It is straightforward to see that these two functors form an isomorphism
  of categories.

  The equivalence of~(b) and~(c) follows from
  Theorem~\ref{thm:piecewiseequivalence}. 
\end{proof}

\begin{remark}
  For an AW*-algebra $A$, recall that $\cC(A)$ is the set of commutative
  AW*-subalgebras, ordered by inclusion. It carries the same
  information as the projection lattice 
  $\Proj(A)$~\cite[Theorem~2.5]{heunen:cstarsubalgebras}. Therefore, everything
  that follows can equivalently be expressed in terms of $\cC(A)$
  instead of $\Proj(A)$. 
\end{remark}

\subsection*{The symmetry group}

If $A$ is a piecewise AW*- algebra, we let $U(A)$ denote the set of
unitary elements of $A$, \ie\ the set of all elements $u \in A$ such
that $u u^* = 1$ (recall that $u \commeas u^*$ for all $u \in A$).
This set carries the structure of a \emph{piecewise group},
\ie\ one can multiply commeasurable elements, the multiplication has a
unit (that is commeasurable with any element), and there is a total
function giving inverses, such that every commeasurable subset
generates a commutative subgroup.
A \emph{piecewise subgroup} is a subset that is a piecewise group in
its own right under the inherited operations (and commeasurability relation).
Every group is a piecewise
group, and conversely, we will be extending the structure of the
piecewise group $U(A)$ to that of a group. Piecewise groups form a
category $\PGroup$ with the evident morphisms.

\begin{definition}
  A \emph{symmetry} in an AW*-algebra $A$ is a self-adjoint unitary
  element; these are precisely the elements of the form $p^\perp - p =
  1-2p$ for some $p \in \Proj(A)$.  Let $U(A)$ denote the group of unitary
  elements of $A$, and define $\Sym(A)$ to be the subgroup of $U(A)$
  generated by the symmetries of $A$. (Notice that if $A$ is not
  commutative then $\Sym(A)$ contains elements that are not
  symmetries.)
\end{definition}

Before moving on to actions of groups on lattices, we consider 
how large the symmetry $\Sym(A)$ group can become. We will see that this
depends on the type: $\Sym(A)$ is (significantly) smaller than $U(A)$ for type
$\mathrm{I}_n$ algebras, and just as large as $U(A)$ for other AW*-algebras.

If $A$ is an AW*-algebra of type $\mathrm{I}_1$, i.e.\ if $A$ is
commutative, then $\Sym(A)$ is as small as possible, namely in bijection
with $\Proj(A)$, as the following example shows.

\begin{example}\label{ex:sym:comm}
  If $A$ is a commutative AW*-algebra,  then the product of symmetries
  is again a symmetry, and so the sets $\Sym(A)$ and $\Proj(A)$ are
  bijective. In fact, $(1-2p)(1-2q) = 1-2((p+q-pq)-pq) = 1-2((p \vee
  q)-(p \wedge q)) = 1-2(p \Delta q)$, where $\Delta$ is the symmetric  
  difference operation. Thus $\Sym(A)$ is the additive group of the
  Boolean ring structure associated to the Boolean algebra $\Proj(A)$.
\end{example}

For AW*-algebras of type $\mathrm{I}_n$ for $n \geq
2$, we will use the fact that traces and determinants are
well-defined for matrices over commutative rings.
Recall that any AW*-algebra of type $\mathrm{I}_n$ takes the form
$\M_n(C)$ for a commutative AW*-algebra
$C$~\cite[Proposition~18.2]{berberian}.
We will use roman letters $a,b,p,\ldots$ for elements of a matrix
algebra $\M_n(B)$ and greek letters $\alpha,\beta,\pi,\ldots$ for
elements of $B$ when both are needed.

\begin{lemma}\label{lem:sym:typeone}
  Let $A=\M_n(C)$ for $n\geq 2$ and a commutative AW*-algebra $C$.
  \begin{enumerate}[\quad (a)]
  \item If $b,c \in C$ satisfy $0 \leq c =b^2 \leq 1$ and $b^*=b$,
    then there exists $u \in \Sym(C)$ with $b=uc_0$,  where $c_0$ is the unique
    positive square root of $c$ in $C$. 
  \item If $u \in U(A)$ has $\det(u)=1$, then $u=(1-2p)(1-2q)$ for some $p,q \in \Proj(A)$.
  \item If $u \in U(A)$ has $\det(u) =1-2\pi$ for $\pi \in \Proj(C)$,
    then $u$ can be written as $u=(1-2p)(1-2q)(1-2r)$ for some $p,q,r \in \Proj(A)$. 
  \item $\Sym(A)$ is the normal subgroup $\{ u \in U(A) \mid \det(u)^2 =
    1 \} = \det^{-1}(\Sym(C))$.
  \end{enumerate}
\end{lemma}
\begin{proof}
  For part (a), observe that the Gelfand spectrum $X$ of $C$ is
  extremally disconnected. So $\inter(b^{-1}(-\infty,0])$ is a clopen
  set, as is its complement $\cl(b^{-1}(0,\infty)))$.
  So the function $u \colon X \to C$ defined by
  \[
  u(x) =
  \begin{cases}
    -1 & \mbox{if $x \in \inter(b^{-1}(-\infty,0])$,} \\
    \phantom{-}1 & \mbox{if $x \in \cl(b^{-1}(0,\infty))$,}
  \end{cases}
  \]
  is continuous. It is clearly a self-adjoint unitary.
  If $x \in \inter(b^{-1}(-\infty,0])$, then $b(x) \leq 0$ and $u(x)=-1$, so
  $b(x)=u(x)c_0(x)$.
  If $x \in \cl(b^{-1}(0,\infty))$, then $b(x) \geq 0$ and $u(x)=1$, so
  $b(x)=u(x)c_0(x)$.
  In either case $b=uc_0$.

  For part (b) we generalize the argument
  of~\cite[page~87]{dye:projections} from matrices with entries in $\C$ to entries in $C$. 
  Let $u \in U(A)$ have determinant 1. Then $u$ is unitarily equivalent to a
  diagonal matrix $\diag(\zeta_1,\ldots,\zeta_n)$ with diagonal
  entries $\zeta_i \in
  U(C)$ satisfying $\prod \zeta_i=1$ \cite{deckardpearcy:diagonal}. Such
  a matrix can be written as $\prod_{i=1}^{n-1}
  \diag(\zeta_{1,i},\ldots,\zeta_{n,i})$, where $\zeta_{i,i} =
  \prod_{k=1}^i \zeta_k$, $\zeta_{i+1,i}=\zeta_{i,i}^*$, and
  $\zeta_{k,i}=1$ otherwise. Therefore, we may assume that
  $u=\diag(\zeta,\zeta^*,1,\ldots,1)$ for fixed $\zeta \in U(C)$.
  Keeping the rest of the matrices involved equal to the identity matrix, we
  may in fact pretend that we are dealing with $n=2$ and
  $u=\diag(\zeta,\zeta^*)$ for fixed $\zeta \in U(C)$.
  We may write $\zeta = \alpha + i\beta$ where
  $\alpha,\beta \in C$ are self-adjoint and satisfy $\alpha^2 + \beta^2 = 1$.

  For each positive $\varphi \in C$, the element $1+\varphi^2$ is invertible in
  $C$, so we can define
  \[
  p_\varphi = \frac{1}{1+\varphi^2} \begin{pmatrix} 1 & \varphi \\ \varphi &
    \varphi^2 \end{pmatrix}.
  \]
  Each $p_\varphi$ is easily seen to be a projection in $A$, so
  $v_\varphi=(1-2p_\varphi)(1-2p_0)$ defines an element of
  $\Sym(A)$. Computing
  \[
    v_\varphi = \frac{1}{1+\varphi^2} \begin{pmatrix} 1-\varphi^2 &
      -2\varphi \\ 2\varphi & 1-\varphi^2 \end{pmatrix}
  \]
  shows that $\det(v_\varphi)=1$ and $\tr(v_\varphi) = 2 \cdot
  \frac{1-\varphi^2}{1+\varphi^2}$. Now, the function $\varphi \mapsto
  \frac{1-\varphi^2}{1+\varphi^2}$ is a composite of an
  order-automorphism $\varphi \mapsto \varphi^2$ of the positive cone
  of $C$ with the Cayley transform $\varphi \mapsto
  \frac{1-\varphi}{1+\varphi}$, which maps  
  the positive cone of $C$ order-anti-isomorphically onto the interval
  $\{ \gamma \in C \mid -1 < \gamma \leq 1 \}$. 
  Hence $\tr(v_\varphi)$ assumes all values in the interval
  $\{ \gamma \in C \mid -2 < \gamma \leq 2 \}$ as $\varphi$ ranges
  over the positive cone of $C$, and actually  
  achieves the value $-2$ by interpreting $p_\infty = \left(\begin{smallmatrix}
  0 & 0 \\ 0 & 1 \end{smallmatrix}\right)$. Diagonalizing $v_\varphi$ to $\diag(\xi,\xi^*)$ 
  with $\xi \in U(C)$, we can therefore make $\tr(v_\varphi) =
  \xi + \xi^* = 2 \Re(\xi)$ assume all values in the positive cone of
  $C$ by varying $\varphi$.

  In particular, for $\zeta = \alpha + i \beta$ as above, there exist positive $\varphi
  \in C$ and $\beta_0 = \sqrt{1 - \alpha^2}$ such that $\zeta_0 = \alpha + i\beta_0 \in
  U(C)$ and $\diag(\zeta_0, \zeta_0^*)$ 
  is unitarily equivalent to $v_\varphi$. Part (a) gives
  $\sigma \in \Sym(C)$ with $\beta = \sigma\beta_0$. The $\R$-linear map
  $\theta$ fixing self-adjoint elements and sending $i$ to $i\sigma$
  defines a $*$-ring automorphism of $C$. 
  Thus $\M_n(\theta)$ is a $*$-ring automorphism of $A$, and 
  $\M_n(\theta)(v_{\varphi})$ is unitarily equivalent to
  $\M_n(\theta)(\diag(\zeta_0, \zeta_0^*)) = \diag(\theta(\zeta_0),
  \theta(\zeta_0)^*) = \diag(\zeta, \zeta^*)$. 
  Because $v_{\varphi}$ is a product of two
  symmetries, the same is true for $\diag(\zeta, \zeta^*)$.


  For part (c), suppose $\det(u)=1-2\pi$. Set $r = \diag(\pi,0) \in
  \Proj(A)$. Notice that $1-2r = \diag(1-2\pi, 1)$ has determinant
  $1-2\pi$. Then $u \cdot (1-2r)$ has determinant~1, so by part~(b) there
  exist $p, q \in \Proj(A)$ such that $u(1-2r) =
  (1-2p)(1-2q)$. Multiplying on the right by $1-2r$, which is its own
  inverse, gives the desired representation of $u$. 

  Finally, part (d) follows from the observation 
  $\Sym(C) = \{ 1-2\pi \mid \pi \in \Proj(C) \}$ and part (c), as follows. Because
  its generators $1-2p$ square to the identity, and the determinant is
  multiplicative, $\Sym(A) \subseteq \{ u \in U(A) \mid \det(u)^2=1
  \}$. Next, if $\det(u)^2=1$, then $\det(u)$ is a symmetry
  in $C$, and hence of the form $1-2\pi$ for some $\pi \in
  \Proj(C)$, so that $\{u \in U(A) \mid \det(u)^2=1 \} \subseteq
  \det^{-1}(\Sym(C))$. Finally, part (c) implies 
  $\det^{-1}(\Sym(C)) \subseteq \Sym(A)$.
\end{proof}

For AW*-algebras of infinite type, it is known that
every unitary is a product of four
symmetries~\cite{thakarebaliga:symmetries}, and therefore the symmetry
group is the full unitary group.

That leaves AW*-algebras of type $\mathrm{II}_1$. For W*-factors of
this type, it is known that $\Sym(A)=U(A)$~\cite{broise:unitaries}.  
If $\Sym(A)$ is closed in $U(A)$, it follows from
from~\cite[Theorem~2]{kadison:unitary}, which holds for AW*-algebras,
that $\Sym(A)=U(A)$. The general question of whether $\Sym(A) = U(A)$
for AW*-algebras $A$ of type $\mathrm{II}_1$  remains open.

\subsection*{Active lattices}

The final piece of structure we will need to be able to recover the
full algebra structure of an AW*-algebra is an action of the
symmetry group.

\begin{definition}
  An \emph{action} of a group $G$ on a piecewise AW*-algebra $A$ is a
  group homomorphism from $G$ to the group of isomorphisms $A \to A$
  in $\PAWstar$.
  Similarly, an action of a group $G$ on a complete orthomodular lattice $P$ is a
  group homomorphism from $G$ to the group of isomorphisms $P \to P$
  in $\cOML$. Explicitly, we can consider a function $G \times P
  \stackrel{\cdot}{\to} P$ satisfying: 
  \begin{itemize}
  \item $1 \cdot p = p$ for all $p \in P$;
  \item $u \cdot (v \cdot p) = (uv) \cdot p$ for all $p \in P$ and $u,v \in G$;
  \item $u \cdot (-) \colon P \to P$ is a morphism of $\cOML$ for each $u \in G$.
  \end{itemize}
  Alternatively, we can speak about a group homomorphism $\alpha \colon G \to \Aut(P)$.
  If the object being acted upon needs to be emphasized, we will speak
  of a \emph{piecewise algebra action} or an \emph{orthomodular action}, respectively.
\end{definition}

If $A$ is an AW*-algebra, then its unitary group $U(A)$ acts on its
projection lattice $\Proj(A)$ by (left) conjugation: if $p$ is a projection
and $u$ is a unitary, then $upu^*$ is again a projection. Moreover,
because conjugation with a unitary is an automorphism of AW*-algebras,
$u(-)u^* \colon \Proj(A) \to \Proj(A)$ is a morphism of complete
orthomodular lattices for each $u \in U(A)$.  
The group $\Sym(A)$ acts on $\Proj(A)$ by restricting the action of $U(A)$. 
This motivates the following definition.

\begin{definition}
  The category $\EAWstar$ of \emph{extended piecewise AW*-algebras} is
  defined as follows. Objects are 4-tuples $(A,P,G,\cdot)$ consisting of:
  \begin{itemize}
  \item a piecewise AW*-algebra $A$;
  \item an object $P$ of $\cOML$ that maps to $\Proj(A)$ under the
    forgetful functor $\cOML \to \PCBoolean$;
  \item a group $G$, that maps to a piecewise subgroup of $U(A)$ under the forgetful functor
    $\Group \to \PGroup$, and that (contains and) is
    generated as a group by the elements $1-2p$ for all $p \in \Proj(A)$;
  \item an action of $G$ on $A$, which restricts to (left) conjugation on $G
    \subseteq A$, that is, $g \cdot h = ghg^{-1}$ for $g \in G$ and $h \in
    G \subseteq A$.
  \end{itemize}
  A morphism $f \colon (A,P,G,\cdot) \to (A',P',G',\cdot')$ is a function $f \colon A \to A'$
  such that:
  \begin{itemize}
  \item $f$ is a morphism of piecewise AW*-algebras;
  \item $f$ restricts to a morphism $P \to P'$ of complete orthomodular lattices;
  \item $f$ restricts to a group homomorphism $G \to G'$;
  \item the equivariance condition $f(u \cdot
    a) = f(u) \cdot' f(a)$ holds for $u \in G$ and $a \in A$.
  \end{itemize}
\end{definition}

In fact, using the equivalence $F \colon \PCBoolean \to \PAWstar$ of
Theorem~\ref{thm:piecewiseequivalence}, we can
whittle the data down further. In particular, if a group $G$ has an
orthomodular action on $P$, there is an induced piecewise algebra
action on $F(P)$ as follows (applying Lemma~\ref{lem:piecewiseorthomodular}
and Theorem~\ref{thm:piecewiseequivalence}):
\[
  G \to \Aut_{\cOML}(P) \subseteq
    \Aut_{\PCBoolean}(P) \cong \Aut_{\PAWstar}(F(P)).
\]
Hence we can reformulate purely in terms of orthomodular
lattices and groups.

\begin{definition}\label{def:activelattive}
  An \emph{active lattice} is a 3-tuple $(P,G,\cdot)$ consisting of:
  \begin{itemize}
  \item a complete orthomodular lattice $P$;
  \item a group $G$, that maps to a piecewise subgroup of $U(F(P))$
    under the forgetful functor $\Group \to \PGroup$, and
    that (contains and) is generated as a group by the elements $1-2p$
    for all $p \in \Proj(F(P)) \cong P$;
  \item an orthomodular action of $G$ on $P$ such that the induced
    piecewise algebra action of $G$ on $F(P)$ restricts to (left)
    conjugation on $G \subseteq F(P)$. 
  \end{itemize}
  A \emph{morphism of active lattices} $(P,G,\cdot) \to
  (P',G',\cdot')$ is a morphism $f \colon P \to P'$ of complete
  orthomodular lattices such that: 
  \begin{itemize}
  \item $Ff$ restricts to a group homomorphism $G \to G'$;
  \item equivariance $f(u \cdot
    p) = Ff(u) \cdot' f(p)$ holds for all $u \in G$ and $p \in P$.
  \end{itemize}
  Active lattices and their morphisms form a category $\AL$.
\end{definition}

\begin{proposition}\label{prop:activelatticesequivalence}
  The categories $\EAWstar$ and $\AL$ are equivalent.
\end{proposition}
\begin{proof}
  We use the unit $\eta_P \colon P \to \Proj(F(P))$ and counit
  $\varepsilon_A \colon F(\Proj(A)) \to A$ isomorphisms of the
  equivalence of Theorem~\ref{thm:piecewiseequivalence} 
  to define appropriate functors.

  Define $G \colon \EAWstar \to \AL$ by
  $G(A,P,G,\alpha)=(P,U(\varepsilon_A^{-1})(G), \alpha \circ
  U(\varepsilon_A))$ and $G(f)=f$. This is well-defined: if $G$ is a
  piecewise subgroup of $U(A)$, then $U(\varepsilon_A^{-1})(G)$ is a
  piecewise subgroup of $U(F(P))$, and precomposing the action
  $\alpha \colon G \to \Aut(P)$ with $U(\varepsilon_A)$ turns it into
  an action of $U(\varepsilon_A^{-1}(G))$ on $P$. The equivariance
  condition on morphisms also follows directly.

  In the reverse direction, define $H \colon \AL \to \EAWstar$ on
  objects by setting 
  \[
      H(P,G,\alpha) = (F(P), \eta_P(P), G, \Aut(\eta_P^{-1}) \circ \alpha)
  \]
  and on morphisms by $H(f)=F(f)$. This is
  well-defined: the structure of $P$ as a 
  complete orthomodular lattice transfers via $\eta_P$ to
  $\eta_P(P)=\Proj(F(P))$, and postcomposing the action $\alpha \colon
  G \to \Aut(P)$ with $\Aut(\eta_P^{-1})$ turns it into an action of $G$ on
  $\Proj(F(P))$. The equivariance condition on morphisms also follows
  directly. 

  Now $\eta_P$ implements a (natural) isomorphism $G \circ H
  (P,G,\cdot) \cong (P,G,\cdot)$, and $\varepsilon_A$ implements a
  (natural) isomorphism $H \circ G(A,P,G,\cdot) \cong
  (A,P,G,\cdot)$. Hence $G$ and $H$ form an equivalence.
\end{proof}

\subsection*{The functor}

We can now define a functor from AW*-algebras to active lattices, and
prove that it is faithful. In Section~\ref{sec:fullness} we will prove that
it is also full. The next proposition tacitly identifies a piecewise AW*-algebra $A$
with $F(\Proj(A))$, as justified by Theorem~\ref{thm:piecewiseequivalence}.

\begin{proposition}\label{prop:thefunctor}
  There is a functor $\ActiveProj \colon \AWstar \to \AL$ acting as
  \[
    \ActiveProj(A) = (\Proj(A), \Sym(A), \cdot),
  \]
  on objects, where $u \cdot p = upu^*$. It sends a morphism $A \to B$ to its restriction $\Proj(A) \to
  \Proj(B)$.  
\end{proposition}
\begin{proof}
  Follows directly from the definitions.
\end{proof}

Via Proposition~\ref{prop:activelatticesequivalence}, we also
write $\ActiveProj$ for the functor $\AWstar \to \EAWstar$.

\begin{lemma}\label{lem:faithful}
  The functor $\ActiveProj$ is faithful.
\end{lemma}
\begin{proof}
  If $\ActiveProj(f)=\ActiveProj(f')$, the continuous linear
  functions $f,f' \colon A \to B$ coincide on $\Proj(A)$. But $A$
  is the closed linear span of $\Proj(A)$.
\end{proof}

The reader might think that Definition~\ref{def:activelattive} could
be reduced further still by considering just complete orthomodular
lattices acted upon by groups generated by them, and letting morphisms
be equivariant pairs of group homomorphisms and morphisms of complete
orthomodular lattices. The following example shows that one cannot
ignore piecewise algebra structure this easily and hope to have a full
and faithful functor out of $\AWstar$. 

\begin{example}
  Consider $\ActiveProj(\M_2(\C)) = (\Proj(\M_2(\C)), \Sym(\M_2(\C)),
  \cdot)$. Define a morphism of complete orthomodular lattices $f
  \colon \Proj(\M_2(\C)) \to \Proj(\M_2(\C))$ by $f(0)=0$, $f(1)=1$,
  and $f(p)=p^\perp$ for $p \neq 0,1$. Recall from
  Lemma~\ref{lem:sym:typeone} that $\Sym(\M_2(\C)) = \{ u \in U_2(\C)
  \mid \det(u)=\pm 1 \}$. Define a group homomorphism $g
  \colon \Sym(\M_2(\C)) \to \Sym(\M_2(\C))$ by $g(u) = \det(u) u$.
  Write $j$ for the injection $\Proj(\M_2(\C)) \to \Sym(\M_2(\C))$
  given by $j(p) = 1-2p$. For $p=0,1$ one easily checks that
  $j(f(p))=g(j(p))$, and for $p \neq 0,1$:
  \[
    j(f(p)) = j(p^\perp) = p - p^\perp = \det(p-p^\perp) \cdot
    (p-p^\perp) = g(p-p^\perp) = g(j(p)).
  \]
  Finally, for $u \in \Sym(\M_2(\C))$ and $p \neq 0,1$:
  \[
    g(u) f(p) g(u)^* = |\det(u)|^2 up^\perp u^* = 1-upu^* =
    f(upu^*),
  \]
  and for $p=0,1$ this formula is also easily seen to hold. Hence $f$
  and $g$ satisfy the equivariance condition. 

  But if there is a linear map $h \colon \M_2(\C) \to \M_2(\C)$ that
  restricts to $f$ on $\Proj(\M_2(\C))$ and to $g$ on $\Sym(\M_2(\C))$,
  then for $\zeta \in U(\C) \backslash \{ \pm 1 \}$, $p \in \Proj(\M_2(\C))
  \backslash\{0,1\}$, and $u=\zeta p+\zeta^*p^\perp \in
  \Sym(\M_2(\C))$, we would have
  \[
    u=g(u)=g(\zeta p+\zeta^*p^\perp)=\zeta
    f(p)+\zeta^*f(p^\perp) = \zeta p^\perp +
    \zeta^*p = u^*,
  \]
  contradicting $\zeta\neq\pm 1$. Therefore it cannot be the case that
  $h$ restricts to $f$. 
\end{example}

In the commutative case, the functor $\ActiveProj$ has nice properties.

\begin{example}
  There is a functor $\CBoolean \to \AL$, that maps a complete Boolean
  algebra $B$ to the active lattice $(B,B_{\text{add}},\cdot)$. Here,
  we identify $B$ with $\Proj(F(B))$ using
  Theorem~\ref{thm:piecewiseequivalence}, and
  $B_{\text{add}}$ is the additive group of $B$ qua Boolean ring,
  which acts trivially on the Boolean algebra $B$ itself. This functor
  is full and faithful. Moreover, it factors through the functor
  $\ActiveProj$. If we restrict to the full subcategory $\cAL$ of
  $\AL$ consisting of the objects $(P,G,\cdot)$ for which $P$ is a
  complete Boolean algebra, then it follows from
  Example~\ref{ex:sym:comm} that the functor $\ActiveProj$ 
  becomes an equivalence of categories. This makes the left triangle
  in the following diagram commute. The other faces obviously commute.
  \[\xymatrix@C-8ex@R+1ex{
    & \cAL \ar@{^{(}->}[rrr] &&& \AL \ar[dl] \\
    \CBoolean \ar@{<-}[ur]^-{\cong} \ar@{^{(}->}|(.475){\hole}[rrr] &&& \cOML \\
    && \cAWstar \ar[uul]_(.2){\ActiveProj} \ar[ull]^-{\Proj}_-{\simeq} \ar@{^{(}->}[rrr] 
    &&& \AWstar \ar[ull]^-{\Proj} \ar[uul]_-{\ActiveProj}
  }\]
\end{example}

\section{Recovering total algebras from piecewise algebras}
\label{sec:fullness}

This section proves that the functor $\ActiveProj$ of
Proposition~\ref{prop:thefunctor} is full. The proof distinguishes two
cases. First, we adapt a theorem of Dye to deal with algebras without
type $\mathrm{I}_2$ summands. Subsequently we deal with algebras of
type $\mathrm{I}_2$ directly. 

\subsection*{Algebras without $\mathrm{I}_2$ summand and a theorem of Dye}

To facilitate the proof of Theorem~\ref{thm:dye} below, we give a
sequence of preparatory lemmas. Several of these are adapted from
Dye's results in~\cite[Section~3]{dye:projections}. Let $A$ be an
AW*-algebra. Any matrix ring $\M_n(A)$ is an AW*-algebra;
see~\cite[Section~62]{berberian}. If $x$ is a row vector in $A^n$
one of whose entries is a projection, then there is a projection in
$\M_n(A)$ whose range is the submodule $Ax \subseteq A^n$ according  
to~\cite[Lemma~2]{dye:projections}. 
We shall refer to these projections in $\M_n(A)$ as \emph{vector projections.}

In particular, given two distinct indices $1\leq i,j \leq n$ and an
element $\alpha \in A$, there is a projection as above where the
vector $x$ is taken to have $1$ in the $i$th entry, $\alpha$ in the 
$j$th entry, and all other entries zero. We denote the corresponding
projection in $\M_n(A)$ by $p_{ij}(\alpha)$. For instance, when $n =
2$, $i = 1$, and $j = 2$, we have 
\[
  p_{12}(\alpha) = \begin{pmatrix}
    (1+\alpha \alpha^*)^{-1} & (1 + \alpha \alpha^*)^{-1}\alpha \\
    \alpha^* (1 + \alpha \alpha^*)^{-1} & \alpha^* (1 + \alpha \alpha^*)^{-1} \alpha
  \end{pmatrix}.
\]
For larger $n$, we follow the convention to only write down the nonzero
2-by-2 parts of such $n$-by-$n$ matrices.
Notice that if $p_{ij}(\alpha) = p_{ij}(\beta)$ for some $\alpha,\beta \in A$,
then $\alpha = \beta$.

\begin{lemma}\label{lem:generatingprojections}
  Let $A$ be an AW*-algebra.
  \begin{enumerate}[(a)]
  \item Sublattices of $\Proj(\M_n(A))$ containing all
    $p_{ij}(\alpha)$ contain all vector projections.
  \item Any projection in $\M_n(A)$ is the supremum of (orthogonal)
    vector projections.
  \end{enumerate}
  Hence the $p_{ij}(\alpha)$ generate $\Proj(\M_n(A))$ as a
    complete orthomodular lattice.
\end{lemma}
\begin{proof}
  Part~(a) is proven as in~\cite[Lemma~7]{dye:projections}.
  For~(b), first note that the proof of~\cite[Lemma~7]{dye:projections} illustrates that
  every nonzero element of $\Proj(\M_n(A))$ contains a nonzero vector projection.
  Fix $p \in \Proj(\M_n(A))$. Zorn's lemma gives a maximal set $S$ of orthogonal
  nonzero homogeneous projections below $p$. We claim that $p$ equals
  $p_0=\bigvee S$. Otherwise $p_0 < p$, so that there would be a nonzero
  vector projection $q \leq p - p_0$. 
  Because $p - p_0 \leq p$, transitivity gives $q \leq p$. Combined with $q \leq p - p_0$,
  this implies $q$ is orthogonal to $p_0$. It follows that $q$ is
  orthogonal to $\Proj(S)$, so $S \sqcup \{q\}$ is an orthogonal set
  of projections below $p$, contradicting maximality.
\end{proof}

We denote by $e_{ij} \in \M_n(A)$ the matrix unit whose $i,j$-entry is
1 and every other entry is zero. Note that $e_{ii} = p_{ij}(0)$ for
any $j \neq i$.
For a projection $p$ in an AW*-algebra, we denote by $s_p = 1 - 2p$ the
corresponding symmetry.

\begin{lemma}\label{lem:unitaryimage}
  Let $A$ and $B$ be AW*-algebras. If $f \colon \Proj(\M_n(A)) \to
  \Proj(\M_n(B))$ is a function satisfying $f(e_{ii})=e_{ii}$ and
  $f(s_p q s_p) = s_{f(p)} f(q) s_{f(p)}$, then for each $i,j$
  and each $\zeta \in U(A)$ there is a unique $\xi \in U(B)$ with $f(p_{ij}(\zeta)) = p_{ij}(\xi)$.
\end{lemma}
\begin{proof}
  Notice that for $\zeta \in U(A)$, we have (in ``2-by-2 shorthand'')
  \[
  p_{ij}(\zeta) = \frac{1}{2}
  \begin{pmatrix}
    1 & \zeta \\
    \zeta^* & 1
  \end{pmatrix}.
  \]
  It is easy to see that
  conjugation by $1-2p_{ij}(\zeta)$ swaps $e_{ii}$ and $e_{jj}$ while
  leaving the remaining diagonal matrix units fixed. Conversely, if $p
  \in \Proj(\M_n(A))$ is such that conjugation by $1-2p$ leaves
  $e_{kk}$ fixed for $k \neq i,j$, then it
  must equal the identity everywhere except in rows and columns $i$
  and $j$. Hence we can write $p=\left(\begin{smallmatrix} \alpha & \beta \\
  \beta^* & \gamma \end{smallmatrix}\right)$ in ``2-by-2 shorthand''. If
  $e_{ii} = (1-2p) e_{jj} (1-2p)$, then $\alpha=\frac{1}{2}$ and
  $\beta^*\beta=\frac{1}{4}$, and it follows from $p=p^2$ that
  $\gamma=\frac{1}{2}$ and $\beta\beta^*=\frac{1}{4}$. Hence
  the projections of the form $p_{ij}(\zeta)$ with $\zeta$ unitary are
  precisely those projections $p$ for which 
  conjugation with $1-2p$ swaps $e_{ii}$ and 
  $e_{jj}$ while leaving the other $e_{kk}$ fixed. 

  Now, because of the assumptions that $f$ sends diagonal matrix
  units to diagonal matrix units, and is equivariant, the same
  statement is true about $f(p_{ij}(\zeta))$. Hence there is some
  unitary $\xi \in U(B)$ such that $f(p_{ij}(\zeta)) =
  p_{ij}(\xi)$; uniqueness follows.
\end{proof}

Recall that a $\C$-linear function $f \colon A \to B$ between
$C^*$-algebras that preserves the involution is a \emph{Jordan
  $*$-homomorphism} if it preserves the Jordan product $a \circ b =
\frac{1}{2}(ab+ba)$; this is readily seen to be equivalent to the
property that $f$ preserves the square of every element.

\begin{lemma}\label{lem:jordan}
  Given a $*$-ring homomorphism $A \to B$ between
  C*-algebras, there is a unique Jordan $*$-homomorphism 
  $A \to B$ that equals it on
  self-adjoint elements.
\end{lemma}
\begin{proof}
  Let $f \colon A \to B$ be a $*$-ring homomorphism.
  As it preserves $1$ it is $\mathbb{Q}$-linear, and it
  follows from preserving positivity that it is in fact $\R$-linear.
  Define complementary projections $q_- = \frac{1}{2}(1+if(i))$ and
  $q_+ = \frac{1}{2}(1-if(i))$ in $B$. Setting
  \begin{align*}
    f_- & \colon A \to q_- B q_- & f_-(a) & =
    \tfrac{1}{2}(f(a) + if(ia)) \\
    f_+ & \colon A \to q_+ B q_+ & f_+(a) & =
    \tfrac{1}{2}(f(a) - if(ia))
  \end{align*}
  gives $*$-ring homomorphisms, where $f_-$ is $\mathbb{C}$-anti-linear and
  $f_+$ is $\mathbb{C}$-linear. Clearly $f=f_+ +f_-$.
  Now define $g \colon A \to B$ by
  \[
  g(a) = f_+(a) + (f_-(a))^*.
  \]
  This $\C$-linear function preserves the involution
  and agrees with $f$ on self-adjoint elements. It is easy to verify
  that it preserves the operation of squaring because the images of
  $f_+$ and $f_-$ are orthogonal in $B$. 
  Uniqueness is straightforward.
\end{proof}

The following lemma records some results of
Dye~\cite{dye:projections} about properties of the ``coordinate
assignment'' from Lemma~\ref{lem:unitaryimage}. 
Basically, it expresses algebraic operations on the coordinates in
lattice-theoretic terms. The subsequent lemma will use these properties to
establish a $*$-ring homomorphism, following~\cite[Lemmas~6 and~8]{dye:projections}.
Recall that a \emph{lattice polynomial} is an expression 
combining a finite number of variables using $\wedge$ and $\vee$;
these are preserved by morphisms in $\cOML$.

\begin{lemma}\label{lem:latticepolynomials}
  There exist lattice polynomials $P$, $Q$, and $R$ such that, for any
  elements $\alpha,\beta,\gamma$ of a C*-algebra $A$ with $\gamma$
  invertible, any integer $n \geq 3$, and any distinct indices
  $1 \leq i,j,k \leq n$, the following hold:
  \begin{enumerate}[\quad (a)]
  \item $p_{ij}(\alpha+\beta) = P\big( p_{ij}(\alpha), p_{ij}(\beta), p_{ik}(\gamma), e_{ii},
    e_{jj}, e_{kk} \big)$;
  \item $p_{ij}(-\alpha\beta) = Q\big(p_{ik}(\alpha), p_{kj}(\beta), e_{ii}, e_{jj}\big)$;
  \item $p_{ij}(-\alpha^*) = R\big (p_{ji}(\alpha), e_{ii}, e_{jj} \big)$.
  \end{enumerate}
\end{lemma}
\begin{proof}
  See~\cite[Lemma~5]{dye:projections},
  \cite[Lemma~4]{dye:projections},
  and~\cite[Lemma~3(i)]{dye:projections}, respectively.
\end{proof}


\begin{lemma}\label{lem:coordinatefunction}
  Let $f \colon \Proj(\M_n(A)) \to \Proj(\M_n(B))$ be a morphism of
  $\cOML$ for AW*-algebras $A,B$, and $n\geq 3$. Suppose 
  $f(e_{ii})=e_{ii}$ for all $i$, and that for any distinct $i,j$ and
  any $\zeta \in U(A)$ there is $\xi \in U(B)$ with $f(p_{ij}(\zeta)) = p_{ij}(\xi)$.
  Then there is a diagonal $W \in U(\M_n(B))$ such that:
  \begin{enumerate}[\quad (a)]
  \item there is a function $\varphi \colon U(A) \to U(B)$ satisfying
    the ``coordinate condition''
    \begin{equation*}
      f(p_{ij}(\alpha)) = W^*p_{ij}(\varphi(\alpha))W 
    \end{equation*}
    for all $\alpha \in U(A)$ and distinct indices $i,j$;
  \item $\varphi$ extends to a $*$-ring homomorphism $A \to B$
    satisfying the coordinate condition for all $\alpha \in A$ and 
    distinct $i,j$. 
  \end{enumerate}
\end{lemma}
\begin{proof}
  Abbreviate the coordinate condition as ($*$).
  By hypothesis, for all indices $j > 1$ there exist $\beta_j \in U(B)$
  such that $f(p_{1j}(1)) = p_{1j}(\beta_j)$. Define
  $W = \diag(1,\beta_2, \dots, \beta_n) \in U(B)$. Then
  $p_{1j}(\beta_j) = W^* p_{1j}(1) W$ for all $j$. Notice
  that conjugation by a diagonal unitary fixes all $e_{ii}$,
  and leaves the set $\{p_{ij}(\alpha)\}$ invariant as $\alpha$
  ranges over $U(A)$ (respectively, over $A$). Thus, replacing
  $f$ with the morphism $p \mapsto Wf(p)W^*$, we may assume
  that $f(p_{1j}(1)) = p_{1j}(1)$ for all $j > 1$, and prove
  that ($*$) holds in both~(a) and~(b) with $W = 1$.
  
  Towards (a), define $\varphi \colon U(A) \to U(B)$ by the condition
  $f(p_{12}(\alpha)) = p_{12}(\varphi(\alpha))$.
  In case $f(p_{1j}(\alpha)) = p_{1j}(\varphi(\alpha))$, for some
  $\alpha \in U(A)$ and distinct $i,j > 1$, it follows
  by applying Lemma~\ref{lem:latticepolynomials} that
  \begin{align*}
    f(p_{ij}(\alpha))
    & = f\big(Q\big( p_{i1}(-1), p_{1j}(\alpha), e_{ii}, e_{jj}\big)\big) \\
    & = f\big(Q\big( R\big( p_{1i}(1), e_{ii}, e_{11} \big),
    p_{1j}(\alpha), e_{ii}, e_{jj} \big) \\
    & = Q\big( R\big( p_{1i}(1), e_{ii}, e_{11} \big),
    p_{1j}(\varphi(\alpha)), e_{ii}, e_{jj} \big) 
    = p_{ij}(\varphi(\alpha)).
  \end{align*}
  In particular, because ($*$) is known to hold in case $i = 1$ and $\alpha = 1$,
  this shows that ($*$) in fact holds for $\alpha = 1$ and any distinct
  $i,j > 1$ (and, of course, when $i = 1$ and $j = 2$).
  Now since ($*$) for the case $\alpha=1$ and $j = 2$, and it holds
  by assumption for $i=1$, $j=2$ and all $\alpha \in U(A)$, then for
  $j > 2$ we find: 
  \begin{align*}
    f(p_{1j}(\alpha))
    & = f\big( Q\big( p_{12}(\alpha), R\big( p_{j2}(1), e_{22}, e_{jj}
    \big), e_{11}, e_{jj} \big) \big) \\
    & = Q\big( p_{12}(\varphi(\alpha)), R\big( p_{j2}(1), e_{22},
    e_{jj} \big), e_{11}, e_{jj} \big) \big) 
      = p_{1j}(\varphi(\alpha)).
  \end{align*}
  Thus the above shows that ($*$) holds for all $\alpha \in U(A)$
  and any $j  \geq 2$. 
  For the remaining case where $i > 1$ and $j = 1$, simply note
  that
  \begin{align*}
    f(p_{i1}(\alpha)) &= f(R\big (p_{1i}(-\alpha^*), e_{ii}, e_{11} \big)) \\
    &= R\big (p_{1i}(-\alpha^*), e_{ii}, e_{11} \big) = p_{i1}(\alpha).
  \end{align*}

  To prove part (b), we start by defining a function $\psi \colon A \to B$.
  Write $\alpha = \alpha_1 + i \alpha_2$ where each $\alpha_k$ is self-adjoint.
  Set
  $\zeta_k=\frac{\alpha_k}{2n} + i\sqrt{1-(\frac{\alpha_k}{2n})^2}$, 
  where $n$ is an integer satisfying $\|\alpha_k\| \leq 2n$ for $k = 1,2$. 
  Then $\zeta_k  \in U(A)$ satisfy $\zeta_k+\zeta_k^*=\frac{\alpha_k}{n}$. 
  Now, an application of Lemma~\ref{lem:latticepolynomials}(a) with $\gamma=1$
  shows $f(p_{ij}(\lambda_1+\lambda_2))=p_{ij}(\mu_1+\mu_2)$ if
  $f(p_{ij}(\lambda_l))=p_{ij}(\mu_l)$, and similarly for sums with more terms. 
  Therefore, in particular, 
  \begin{align*}
    f(p_{ij}(\alpha/n)) &= f(p_{ij}(\zeta_1+\zeta_1^*+i\zeta_2+i\zeta_2^*)) \\
    &= p_{ij}(\varphi(\zeta_1)+\varphi(\zeta_1^*)+\varphi(i\zeta_2)+\varphi(i\zeta_2^*)).
  \end{align*}
  Taking $\beta$ to be $n$ times the argument of $p_{ij}$ in the previous line, we have
  $\beta \in B$ with $f(p_{ij}(\alpha))=p_{ij}(\beta)$. Setting $\psi(\alpha)=\beta$
  yields $f(p_{ij}(\alpha))=p_{ij}(\psi(\alpha))$ for all $\alpha \in A$. 
  It follows that $p_{ij}(\psi(\alpha))=p_{ij}(\varphi(\alpha))$ for
  unitary $\alpha$, whence $\psi$ extends $\varphi$.

  Next we prove that $\psi$ is a $*$-ring homomorphism. First apply 
  Lemma~\ref{lem:latticepolynomials}(a) with $\gamma=1$ and use part (a)
  to deduce 
  \begin{align*}
    p_{ij}\big(\psi(\alpha)+\psi(\beta)\big) 
    & = P\big( p_{ij}(\psi(\alpha)), p_{ij}(\psi(\beta)), p_{ik}(\psi(1)), e_{ii}, e_{jj}, e_{kk} \big)\\
    & = P\big( f(p_{ij}(\alpha)), f(p_{ij}(\beta)), f(p_{ik}(1)), f(e_{ii}), f(e_{jj}), f(e_{kk}) \big)\\
    & = f\big(P\big(p_{ij}(\alpha), p_{ij}(\beta), p_{ik}(1), e_{ii}, e_{jj}, e_{kk}\big)\big) \\
    & = f(p_{ij}(\alpha+\beta)) = p_{ij}\big(\psi(\alpha+\beta)\big),
  \end{align*}
  and conclude that $\psi$ is additive. Hence also
  $\psi(0)=\psi(0+0)-\psi(0)=0$. 
  It similarly follows from Lemma~\ref{lem:latticepolynomials}(b) 
  that $\psi$ is multiplicative. 
  Finally, Lemma~\ref{lem:latticepolynomials}(c) shows that $\psi$
  preserves the involution. 
\end{proof}
 
The assumption that each $\zeta \in U(A)$ allows $\xi \in U(A)$
such that $f(p_{ij}(\zeta)) = p_{ij}(\xi)$ is slightly stronger than necessary
and is only used to shorten the proof above. With more work,
one may simply assume that this is the case when $i = 1$ and $j = 2$.

We are now ready to prove an AW*-analogue of Dye's
theorem~\cite[Theorem~1]{dye:projections}.

\begin{theorem}\label{thm:dye}
  Let $A$ be an AW*-algebra without type $\mathrm{I}_2$ summands,
  and $B$ any AW*-algebra. A $\cOML$-morphism $f
  \colon \Proj(A) \to \Proj(B)$ extends to a Jordan $*$-homomorphism 
  $A \to B$ if and only if $f(s_p q s_p) = s_{f(p)} f(q) s_{f(p)}$.
\end{theorem}
\begin{proof}
  The ``only if'' direction follows because the expression to be
  preserved can be written in terms of Jordan operations:
  \begin{align*}
    s_p q s_p = (1-2p) q (1-2p) 
    &= q - 2 (pq+qp) + 4 pqp \\
    &= q - 2 (p \circ q) + 4 (pqp + p^{\perp} q p^{\perp}) \circ p \\
    &= q - 2 (p \circ q) + 4 (p+q-1)^2 \circ p.
  \end{align*}
  For the converse we first reduce the problem to the case 
  $A=\M_n(D)$ for $n \geq 3$ and AW*-algebra $D$. Indeed,
  \cite[Theorem~15.3]{berberian} and~\cite[Theorem~18.4]{berberian}
  provide unique orthogonal central projections
  $p_1,p_2,\ldots,p_\infty$ with sum $1$ such that $p_nA$ is of
  type~$\mathrm{I}_n$ for $n < \infty$ and $p_\infty A$ has no finite
  type~$\mathrm{I}$ summands. Then $A$ is the C*-sum of
  $p_iA$~\cite[Proposition~10.2]{berberian}, 
  and it suffices to consider one summand at a
  time because Jordan $*$-homomorphisms are closed under direct sums.
  By~\cite[Exercise~19.2]{berberian}, $p_\infty A \cong \M_3(D)$
  for some AW*-algebra $D$. For each finite $n$,
  by~\cite[Proposition~18.2]{berberian}, $p_nA \cong \M_n(C)$ for some
  commutative AW*-algebra $C$.
  By assumption $p_2=0$, leaving us with commutative AW*-algebras
  $p_1A$. But this case is taken care of by the
  duality~\eqref{eq:commutativeduality}, since morphisms in $\cAWstar$
  are definitely Jordan $*$-homomorphisms.
  Thus we may replace $A$ with $\M_n(A)$ for $n \geq 3$.

  Next, we make another reduction (replicating the proof
  of~\cite[Theorem~1]{dye:projections}) to show that we may also
  replace $B$ with $\M_n(B)$. Any two distinct diagonal matrix units
  $e_{ii}$ in $\M_n(A)$ have a common complement, so the same is true
  for their images under
  $f$. By~\cite[Theorem~6.6]{kaplansky:awstar} this means 
  that their images $f(e_{ii})$ are equivalent
  projections. These $n$ orthogonal equivalent projections sum
  to $1$, so by~\cite[Proposition~16.1]{berberian} they form the diagonal
  projections of a set of $n$-by-$n$ matrix units in $B$. Thus we may
  replace $B$ by $\M_n(B)$ and assume that $f(e_{ii}) = e_{ii}$. 

  So we are assuming that $A$ and $B$ are AW*-algebras with an
  $\cOML$-morphism $f \colon \Proj(\M_n(A)) \to \Proj(\M_n(B))$ for
  $n \geq 3$. 
  Combining Lemmas~\ref{lem:unitaryimage}
  and~\ref{lem:coordinatefunction} produces a $*$-ring homomorphism
  $\varphi \colon A \to B$ and a diagonal $W \in U(M_n(B))$ such that 
  $f(p_{ij}(\alpha))=W^* p_{ij}(\varphi(\alpha)) W$ for all $\alpha
  \in A$ and all distinct $i,j$.
  It follows from the definition of $p_{ij}$ that
  $f(p_{ij}(\alpha))=W^* \big(\M_n\varphi(p_{ij}(\alpha))\big) W$ for all
  $i,j$ and $\alpha \in A$.

  Next we show that $\varphi$ preserves suprema of projections, using an
  auxiliary function $\pi \mapsto p_{12}(\pi) \wedge e_{22}$.  
  It is a well-defined morphism $j_A \colon \Proj(A) \to \Proj(M_n(A))$ of
  complete orthomodular lattices that is injective. Hence
  \begin{align*}
    j_B(\bigvee_i \varphi(\pi_i))
    & = \bigvee_i p_{12}(\varphi(\pi_i)) \wedge e_{22}  
    = \bigvee_i W f(j_A(\pi_i)) W^* \\
    & = W f(j_A(\bigvee_i \pi_i)) W^* 
    = p_{12}(\varphi(\bigvee_i \pi_i)) \wedge e_{22} 
    = j_B(\varphi(\bigvee_i \pi_i)),
  \end{align*}
  and so $\bigvee_i \varphi(\pi_i) = \varphi(\bigvee_i \pi_i)$ by
  injectivity of $j_B$.

  Consequently, the $*$-ring homomorphism $\M_n(\varphi) \colon
  \M_n(A) \to \M_n(B)$ also preserves suprema of projections
  by~\cite[Theorem~8.2 and 
  Remark~8.3]{heunenreyes:diagonal}.  
  Hence so does its conjugation with $W$.
  Now Lemma~\ref{lem:generatingprojections} guarantees that
  $W^* \M_n(\varphi) W$ equals $f$ on all of $\Proj(\M_n(A))$. 
  The proof is concluded by an application of Lemma~\ref{lem:jordan}.
\end{proof}

\begin{remark}
  It remains an open question whether every morphism of complete orthomodular
  lattices $\Proj(A) \to \Proj(B)$ extends to a Jordan $*$-homomorphism
  $A \to B$ when $A$ and $B$ are AW*-algebras and $A$ has no type 
  $\mathrm{I}_2$ summands. This is known to be the case when $A$ and $B$
  are W*-algebras~\cite[Corollary~1]{buncewright:jordan}. Our proof
  suffices to answer this question for AW*-algebras if
  Lemma~\ref{lem:unitaryimage} holds more generally without the
  equivariance assumption. 

  The analogous generalization of Lemma~\ref{lem:unitaryimage} is known
  to hold over a von Neumann regular ring $R$, i.e.\  a ring such that every
  $a \in R$ admits $b \in R$ with $a=aba$. In this setting, denote by
  $q_{ij}(\alpha)$ the idempotent in $\M_n(R)$ whose row range is the
  submodule of $R^n$ generated by the row vector with $i$th entry $1$ and
  $j$th entry $\alpha$. Then the $q_{ij}(\alpha)$ for invertible $\alpha$
  are characterised in lattice-theoretic terms as those projections $p$ that
  complement both $e_{ii}$ and $e_{jj}$, i.e.\ $p \wedge e_{ii}=0
  =p \wedge e_{jj}$ and $p \vee e_{ii} = e_{ii}+e_{jj} = p\vee e_{jj}$
  (see Part~II, Chapter~III, Lemma~3.4 of~\cite{vonneumann:continuous}).

  Unfortunately, this characterisation does not hold for a general AW*-algebra $A$.
  To see the difficulty, let $\alpha \in A$ be neither a left nor a right zerodivisor, but also not invertible. 
  Considering $A^2$ as a left $\M_2(A)$-module, $p=p_{21}(\alpha)$ is a projection 
  with range $A \begin{pmatrix} \alpha & 1 \end{pmatrix}$. Since $\alpha$ is not a left 
  zerodivisor, $A \begin{pmatrix} \alpha & 1 \end{pmatrix} \cap A \begin{pmatrix} 0 & 1 \end{pmatrix} = 0$,
  whence $\range(p \wedge e_{22})=\range(p) \cap \range(e_{22})=0$, so 
  $p \wedge e_{22}=0$. Similarly $p \wedge e_{11}=0$. Furthermore, $p^\perp$ has range 
  $A \begin{pmatrix} 1 & -\alpha^* \end{pmatrix}$, which has zero intersection with
  $A \begin{pmatrix} 1 & 0 \end{pmatrix}$ because $\alpha$ is not a right zerodivisor, 
  so that $p^\perp \wedge e_{11}=0$, which $(-)^\perp$ sends to $p \vee e_{22} = 1$. 
  Also $p \vee e_{11}=1$, so $p$ complements both $e_{11}$ and $e_{22}$.
  However, because $\alpha$ is not invertible, it cannot be of the form 
  $p=p_{12}(\beta)$~\cite[Lemma~3(ii)]{dye:projections}. 
\end{remark}

\begin{lemma}\label{lem:fullness:typeone}
  Let $A$ and $B$ be AW*-algebras.  If $f \colon \ActiveProj(A) \to
  \ActiveProj(B)$ is a morphism in $\EAWstar$ such that $f$ extends to a
  continuous $\C$-linear function $g \colon A \to B$, then 
  $g$ is a morphism of AW*-algebras satisfying $\ActiveProj(g)=f$.
\end{lemma}
\begin{proof}
  We first show that $g(a)g(b) = g(ab)$ for all $a, b \in A$. 
  Because the functions $A \times A \to A$ on each side of the
  equation above are continuous and bilinear, and 
  because $A$ is the closed linear span of its projections,     
  it suffices to consider the case where $a$ and $b$ are
  projections. 
  Now, for $p,q \in \Proj(A)$,
  \begin{align*}
    1-2g(p)-2g(q)+4g(pq)
    & = g(1-2p-2q+4pq) \\
    & = f((1-2p)(1-2q)) \\
    & = f(1-2p) f(1-2q) \\
    & = (1-2f(p))(1-2f(q)) \\
    & = 1-2g(p)-2g(q)+4g(p)g(q),
  \end{align*}
  and therefore $g(pq)=g(p)g(q)$ as desired. The above equations used,
  respectively: linearity of $g$; $g$ extends $f$; $f$ is a group
  homomorphism on $\Sym(A)$; $f$ is a piecewise algebra morphism; $g$
  extends $f$.

  So $g$ is an algebra homomorphism, and it is readily seen to be a
  $*$-homomorphism using linearity and the fact that it equals $f$ on
  normal elements. Because $f$ preserves suprema of
  projections and $g$ extends it, we see that $g$ is a morphism in
  \AWstar, which obviously satisfies $\ActiveProj(g) = f$.
\end{proof}

\begin{corollary}\label{cor:fullness:typeone}
  Let $A$ and $B$ be AW*-algebras, and $f \colon
  \ActiveProj(A) \to \ActiveProj(B)$ a morphism of $\EAWstar$. If $A$
  has no type $\mathrm{I}_2$ summand, $f$ is in the
  image of $\ActiveProj$. 
\end{corollary}
\begin{proof}
  Theorem~\ref{thm:dye} extends $f \colon \Proj(A) \to \Proj(B)$ 
  to a Jordan $*$-homomorphism $g \colon A \to B$, which is
  continuous~\cite[Page~439]{stoermer:jordan}. Because $A$ is the
  closed linear span of $\Proj(A)$, 
  in fact $f$ and $g$ coincide as functions $N(A) \to N(B)$.
  Hence the result follows from Lemma~\ref{lem:fullness:typeone}.
\end{proof}

\subsection*{Type $I_2$ algebras}

Next we focus on algebras of type $\mathrm{I}_2$. As in
Lemma~\ref{lem:sym:typeone}, we will use the fact that determinants
and traces are well-defined for matrices with entries in a commutative ring.

\begin{proposition}\label{prop:fullness:typeonetwo}
  Let $A$ and $B$ be AW*-algebras, and $f \colon \ActiveProj(A) \to
  \ActiveProj(B)$ a morphism of $\EAWstar$. If $A$ is type
  $\mathrm{I}_2$, then $f$ is in the image of
  $\ActiveProj$.
\end{proposition}
\begin{proof}
  Let $C$ be a commutative AW*-algebra with
  $A=\M_2(C)$; this exists by~\cite[Proposition~18.2]{berberian}.   
  The algebra $C$ is embedded in $A=\M_2(C)$ by $\gamma \mapsto \diag(\gamma,\gamma)$.
  Fix $p = e_{11} \in \Proj(A)$ and $u = e_{12} + e_{21} \in \Sym(A)$. 
  Since $upu=p^{\perp}$, we deduce
  \begin{align*}
    f(u) f(p) f(u) &= f(p)^{\perp}, \\
    f(u) f(p) &= f(p)^{\perp} f(u), \\
    f(p) f(u) &= f(u) f(p)^{\perp}. 
  \end{align*}
  It follows that $e'_{11} = f(p)$, $e'_{12} = f(p) f(u)$, $e'_{21} = f(u) f(p)$,
  and $e'_{22} = f(p)^{\perp}$ form a self-adjoint set of 2-by-2
  matrix units in $B$ (see~\cite[Page~429]{kadisonringrose}).
  The image $D = f(C) \subseteq B$ is a commutative $*$-subalgebra
  centralizing all $e'_{ij}$. Letting $V \subseteq B$ be the $D$-span
  of the $e'_{ij}$, it follows that $V$ is a $*$-subalgebra of $B$ isomorphic to $\M_2(D)$.
  Define a $C$-linear function $g \colon A \to V \subseteq B$ by $g(e_{ij}) = e'_{ij}$
  and $g(\gamma)=f(\gamma)$ for $\gamma \in C$; it is a $*$-homomorphism.

  Next we will prove that $g$ equals $f$ on all $q \in \Proj(A)$. Notice
  that $\det(q)$ is a projection in $C$. 
  Using properties of the adjugate matrix~\cite[XIII.4.16]{Lang} we find
  $\det(q) 1_A = \adj(q) q = \adj(q) q^2 = \det(q) q$, and so
  $\det(q) 1_A \leq q$ in $\Proj(A)$. 
  Because $\adj \colon \M_2(C) \to \M_2(C)$
  is $C$-linear, the projection $q' = q - \det(q)1_A$ has
  determinant
  \[
    \det(q') = \adj(q') q' = (\adj(q) - \det(q)1_A)(q - \det(q)1_A) = 0.
  \]
  So without loss of generality we may assume $\det(q) = 0$.
  In this case one can compute that $\tau = \tr(q)$ is a projection of $C$.
  As any projection in $A$ with trace $\tau$ and determinant
  zero, $q$ can be written (in standard matrix units $e_{ij}$) in the form
  \[
  q=\frac{1}{2}
  \begin{pmatrix}
    \tau + \alpha & \zeta \beta \\
    \zeta^* \beta & \tau - \alpha
  \end{pmatrix}
  \]
  where $\alpha, \beta \in C$ are self-adjoint, satisfy $\alpha^2 + \beta^2 = \tau$, and
  $\zeta \in C$ is a partial isometry with $\zeta \zeta^* \beta = \beta$.
  Replacing $\zeta$ with $\zeta + (1- \zeta \zeta^*)$ if necessary,
  we may in fact assume $\zeta \in U(C)$.
  Because the algebra $C$ has square
  roots~\cite[Corollary~2.3]{deckardpearcy:diagonal}, 
  there exists $\xi \in
  U(C)$ such that $i\xi^2 = \zeta$.
  From $\alpha^2 + \beta^2 = \tau$ one deduces that $\tau$ supports
  $\alpha$ and $\beta$, 
  so $\tau^{\perp}$ annihilates $\alpha$
  and $\beta$. Then
  \begin{align}
    1 - 2q &= 
    \begin{pmatrix}
      \tau^{\perp}-\alpha & -\zeta \beta \\
      -\zeta^* \beta & \tau^{\perp}+\alpha
    \end{pmatrix} \tag{$*$} \\
    &=
    \begin{pmatrix}
      \tau - \tau^{\perp} & 0 \\
      0 & 1
    \end{pmatrix}
    \begin{pmatrix}
      -\tau^{\perp}-\alpha & -i \xi^2 \beta \\
      i(\xi^*)^2 \beta & \tau^{\perp}+\alpha
    \end{pmatrix} \notag \\
    &=
    \begin{pmatrix}
      \tau - \tau^{\perp} & 0 \\
      0 & 1
    \end{pmatrix}
    \begin{pmatrix}
      -\xi & 0 \\
      0 & \xi^*
    \end{pmatrix}
    \begin{pmatrix}
      \tau^{\perp} + \alpha & i\beta \\
      i\beta & \tau^{\perp} + \alpha
    \end{pmatrix}
    \begin{pmatrix}
      \xi^* & 0 \\
      0 & \xi
    \end{pmatrix} \notag \\ \notag
    &= \big( (\tau - \tau^{\perp})p + p^{\perp} \big) 
          \big( -\xi p + \xi^* p^{\perp} \big) 
          \big( (\tau^{\perp}+\alpha)1 + i\beta u \big)
          \big( \xi^* p + \xi p^{\perp} \big). 
  \end{align}
  The four factors in the right hand side are elements of $\Sym(A)$ by
  Lemma~\ref{lem:sym:typeone}(d). Because $f$ is piecewise linear and is multiplicative
  when restricted to $\Sym(\M_2(C))$, applying $f$ to~($*$) and invoking
  piecewise linearity gives
  \begin{align*}
    1-2f(q)
    & = f\big( (\tau - \tau^{\perp})p + p^{\perp} \big)
           f\big( -\xi p + \xi^* p^\perp \big) 
           f\big( (\tau^{\perp}+\alpha)1+i\beta u \big) 
           f\big( \xi^* p + \xi p^\perp \big) \\
    & = (\tau^{\perp}-\alpha) f(p)  
           - \zeta \beta f(p)f(u)  
           - \zeta^* \beta f(u)f(p) 
           + (\tau^{\perp}+\alpha) f(p)^\perp \\
    & = (\tau^{\perp}-\alpha) g(e_{11})  
           - \zeta \beta g(e_{12})  
           - \zeta^* \beta g(e_{21}) 
           + (\tau^{\perp}+\alpha) g(e_{22}) \\
    & = g(1-2q) 
       = 1 - 2g(q),
  \end{align*}
  whence $f(q) = g(q)$. 
  Finally, because $*$-homomorphisms are continuous, an application
  of Lemma~\ref{lem:fullness:typeone} finishes the proof.
\end{proof}

\subsection*{Fullness of the functor and some open problems}

We summarize by showing that $\ActiveProj
\colon \AWstar \to \AL$ is a full functor.

\begin{theorem}\label{thm:fullness}
  If $A$ and $B$ are AW*-algebras, and $f \colon \ActiveProj(A) \to
  \ActiveProj(B)$ is a morphism in $\AL$, then $f=\ActiveProj(g)$ for
  some $g \colon A \to B$ in $\AWstar$.
\end{theorem}
\begin{proof}
  Proposition~\ref{prop:activelatticesequivalence} allows us to replace
  $\AL$ by $\EAWstar$. As any AW*-algebra, $A$ is a direct sum $A = p_1A \oplus p_2A$
  for central projections $p_1$ and $p_2 = 1-p_1$, where $p_1 A$ is
  a type $\mathrm{I}_2$ AW*-algebra and $p_2 A$ is an AW*-algebra
  without type $\mathrm{I}_2$ summands~\cite[Section~15]{berberian}.
  Because $p_i$ are central in $A$, the symmetries $1-2p_i$ are central in $\Sym(A)$.
  So the projections $q_i = f(p_i)$ are central in $B$, as the symmetries $1-2q_i$ are central in $\Sym(B)$ because $f$ is a morphism in $\AL$.
  Thus $f$ restricts to two morphisms $f_i
  \colon \ActiveProj(p_i A) \to \ActiveProj(q_i B)$ of $\EAWstar$.
  Corollary~\ref{cor:fullness:typeone} provides a  
  morphism $g_1 \colon p_1 A \to q_1 B$ of $\AWstar$ with $\ActiveProj(g_1)=f_1$, and 
  Proposition~\ref{prop:fullness:typeonetwo} provides a morphism $g_2 \colon
  p_2 A \to q_2 B$ of $\AWstar$ with $\ActiveProj(g_2)=f_2$. Their sum
  $g \colon A \to B$, defined by $g(a) = g_1(p_1 a) + g_2(p_2 a)$, is a
  morphism of $\AWstar$ satisfying $\ActiveProj(g)=f$.
\end{proof}

\begin{corollary}
  $\AWstar$ is equivalent to a full subcategory of $\AL$.
\end{corollary}
\begin{proof}
  Follows directly from Lemma~\ref{lem:faithful} and Theorem~\ref{thm:fullness}.
\end{proof}

This corollary immediately presents the problem of characterizing those
active lattices in the essential image of $\ActiveProj$. That is, for which
active lattices $(P,G,\cdot)$ does there exist an AW*-algebra $A$
such that $(P,G,\cdot) \cong \ActiveProj(A)$ as active lattices?
This is a \emph{coordinatization problem,} reminiscent of von Neumann's
coordinatization of continuous geometries by continuous regular
rings~\cite{vonneumann:continuous}. The authors are currently
unaware of any active lattices that are not in the essential image
of $\ActiveProj$. A solution to this problem should provide deeper
insight into how exactly the active lattice $\ActiveProj(A)$ ``encodes''
the ring structure of an AW*-algebra $A$.

We incorporated the symmetry group into $\ActiveProj(A)$ to
circumvent the problem that the product $pq$ of two projections in an
AW*-algebra $A$ is only a projection if $p$ and $q$ commute.
Another way to bypass this shortcoming would be to consider the submonoid
$P(A) \subseteq A$ generated by $\Proj(A)$.
The involution of $A$ restricts to $P(A)$, and this makes $A$ into a
\emph{Baer $*$-semigroup} in the sense of Foulis~\cite{foulis:baer}
(that is, a $*$-semigroup in which the right annihilator of any subset is generated
as a right ideal by a projection).
The assignment $A \mapsto P(A)$ is a functor from $\AWstar$ to the category of
Baer $*$-semigroups with morphisms given by $*$-homomorphisms that preserve
annihilating projections.
This functor is faithful for the same reason given in the proof of Lemma~\ref{lem:faithful}.
Theorem~\ref{thm:fullness} now suggests the natural question: is this functor also full?

In conclusion, our results also suggest the following natural question for
general C*-algebras: can one reconstruct a C*-algebra $A$ from the
piecewise C*-algebra $N(A)$, the unitary group $U(A)$, and the action
by conjugation of the latter on the former?
The following proposition shows that this comes down to a Mackey--Gleason type
problem once again.

\begin{proposition}
  Let $A,B$ be C*-algebras, and $f \colon N(A) \to N(B)$ 
  a morphism of piecewise C*-algebras that restricts to a group homomorphism
  $U(A) \to U(B)$. Then $f$ extends to a $*$-homomorphism $A \to B$
  if and only if it is additive on self-adjoints.
\end{proposition}
\begin{proof}
  One direction is obvious. For the other, assume that $f$ is
  additive on self-adjoints. Since any $a \in A$ can be written as
  $a=a_1+a_2$ for self-adjoint $a_1=\tfrac{1}{2}(a+a^*)$ and
  $a_2=\tfrac{1}{2i}(a-a^*)$, if $f$ extends to a linear function $f
  \colon A \to B$, then it does so uniquely, by
  $f(a)=f(a_1)+if(a_2)$. First notice that this is well-defined and
  coincides with the given values for $a \in N(A)$, since in that case $a_1
  \commeas a_2$. 

  As $(a+b)_1 = a_1+b_1$ and $(a+b)_2 = a_2+b_2$,
  the assumption makes the extension $f \colon A \to B$
  additive. Next, for $z \in \mathbb{C}$, say $z=x+iy$ for real $x$
  and $y$, compute
  \begin{align*}
    f(za) & = f(xa_1 - ya_2) + if(xa_2 + ya_1), \\
    zf(a) & = f(xa_1) - f(ya_2) + if(xa_2) + if(ya_1).
  \end{align*}
  So the assumption in fact makes the extension $f \colon A \to B$ a
  $\mathbb{C}$-linear function. It also clearly satisfies
  $f(a^*)=f(a)^*$ and $f(1)=1$. 

  Finally, any self-adjoint $a \in A$
  can be written as $a=\tfrac{1}{2}a_++\tfrac{1}{2}a_-$ for unitaries
  $a_\pm = a \pm i\sqrt{1-a^*a}$ that commute with each
  other and $a$. Therefore any element of $A$ is a linear combination
  of four unitaries. Because $f$ restricts to a homomorphism $U(A) \to
  U(B)$, it therefore preserves multiplication on all of $A$.
\end{proof}

One could take into account the action by conjugation of $U(A)$ on
$N(A)$, but it is not clear at all how additionally assuming that
$f(uau^*)=f(u)f(a)f(u)^*$ should guarantee that $f$ is additive on
self-adjoints.

\bibliographystyle{amsplain}
\bibliography{awstar}

\end{document}